\documentclass[leqno,final]{siamltex}

\usepackage{amsmath,amssymb,amsfonts,amscd,amsxtra}
\usepackage{braket,amsfonts}
\usepackage{latexsym}

\usepackage{array}

\usepackage{xspace}
\usepackage{bold-extra}
\usepackage[most]{tcolorbox}

\colorlet{texcscolor}{blue!50!black}
\colorlet{texemcolor}{red!70!black}
\colorlet{texpreamble}{red!70!black}
\colorlet{codebackground}{black!25!white!25}

\patchcmd\newpage{\vfil}{}{}{}
\numberwithin{equation}{section}

\renewcommand{\theequation}{\thesection.\arabic{equation}}

\newtheorem{rem}{Remark}

\renewcommand{\Im}{{\mbox{Im}}}
\renewcommand{\Re}{{\mbox{Re}}}

\newcommand{\norm}[1]{\left\Vert#1\right\Vert}
\newcommand{\abs}[1]{\left\vert#1\right\vert}

\title{A mathematical theory for Fano resonance in a periodic array of narrow slits
  \thanks{J. Lin was partially supported by the NSF grant DMS-1417676,
and H. Zhang was supported by HK RGC grant GRF 16304517 and GRF 16306318.}}

\author{Junshan Lin
  \thanks{Department of Mathematics and Statistics, Auburn University, Auburn, AL 36849  (\tt jzl0097@
auburn.edu ).}%
  \and
  Stephen P. Shipman
  \thanks{Department of Mathematics, Louisiana State University, Baton Rouge, LA 70803 (\tt shipman@lsu.edu).}
  \and
 Hai Zhang%
  \thanks{Department of Mathematics, 
 HKUST,  Clear Water Bay, Kowloon, Hong Kong
    (\tt haizhang@ust.hk).}
}

\begin{document}
\maketitle

\begin{abstract}
This work concerns resonant scattering by a perfectly conducting slab with periodically arranged subwavelength slits, with two slits per period.  There are two classes of resonances, corresponding to poles of a scattering problem.  A sequence of resonances has an imaginary part that is nonzero and on the order of the width $\varepsilon$ of the slits; these are associated with Fabry-Perot resonance, with field enhancement of order $1/\varepsilon$ in the slits.  The focus of this study is another class of resonances which become real valued at normal incidence, when the Bloch wavenumber $\kappa$ is zero.  These are embedded eigenvalues of the scattering operator restricted to a period cell, and the associated eigenfunctions extend to surface waves of the slab that lie within the radiation continuum.  When $0<|\kappa|\ll 1$, the real embedded eigenvalues will be perturbed as
complex-valued resonances, which induce the Fano resonance phenomenon. 
We derive the asymptotic expansions of embedded eigenvalues and their perturbations as resonances
when the Bloch wavenumber becomes nonzero. Based on the quantitative analysis of the diffracted field, 
we prove that the Fano-type anomalies occurs for the transmission of energy through the slab, and show that the
field enhancement is of order $1/(\kappa\varepsilon)$, which is stronger than Fabry-Perot resonance.
\end{abstract}




\setcounter{equation}{0}
\setlength{\arraycolsep}{0.25em}
\section{Introduction}\label{sec:introduction}
Fano resonance, which was initially recognized in the study of the autoionizing states
of atoms in quantum mechanics \cite{fano},  is a type of resonant scattering that gives rise to asymmetric spectral line shapes. 
Fano resonance has been extensively explored more recently
in photonics due to its unique resonant feature of a sharp transition from peak to dip in the transmission signal,
which leads to the design of efficient optical switching devices and photonic devices with high quality factors. We refer the readers to \cite{hsu1, limonov, lukyanchuk} and the references therein for detailed discussions. Mathematically, Fano resonance can be attributed to the perturbation of certain eigenvalues embedded in the continuous (radiation)
spectrum of the underlying differential operators and the corresponding bound states in the continuum (sometimes called BIC)~\cite{hsu2,hsu3,shipman05,lu2}. 
The existence of embedded eigenvalues in photonic slab structures using symmetry
is rigorously established in \cite{bonnet_starling94}. It is known that under smaller perturbation which destroys the symmetry, the embedded eigenmode disappears as its frequency becomes a complex resonance pole, and the transmission coefficient across the slab exhibits a sharp asymmetric shape \cite{shipman05, shipman10}. 
Guided-mode theory~\cite{fan}, analytic perturbation theory~\cite{shipman05,shipman13}, and an augmented scattering matrix method~\cite{nazarov18} have been used to study the Fano resonant transmission for several configurations when the Bloch wave vector is perturbed and the bound states associated with the embedded eigenvalues become quasi-modes.

For photonic structures, the quantitative studies of embedded eigenvalues and Fano resonance mostly rely on numerical approaches. 
In this paper, we derive explicitly the embedded eigenvalues and provide quantitative analysis of their perturbations as resonances in the context of periodic metallic structures with small holes. Based on these explicit expressions, we are able to obtain the transmission and reflection for the scattering of the metallic structure,
which allow us to prove the appearance of Fano-type resonance phenomenon rigorously.
The study in this paper is also in line with our recent attempt to understand the so-called extraordinary optical transmission (EOT) in
subwavelength hole structures \cite{ebbesen98, garcia10}.
We also  refer to~\cite{eric10-2, eric10} for resonant scattering by closely related cavity structures in perfect conductors.

\begin{figure}
\begin{center}
\includegraphics[height=4.5cm]{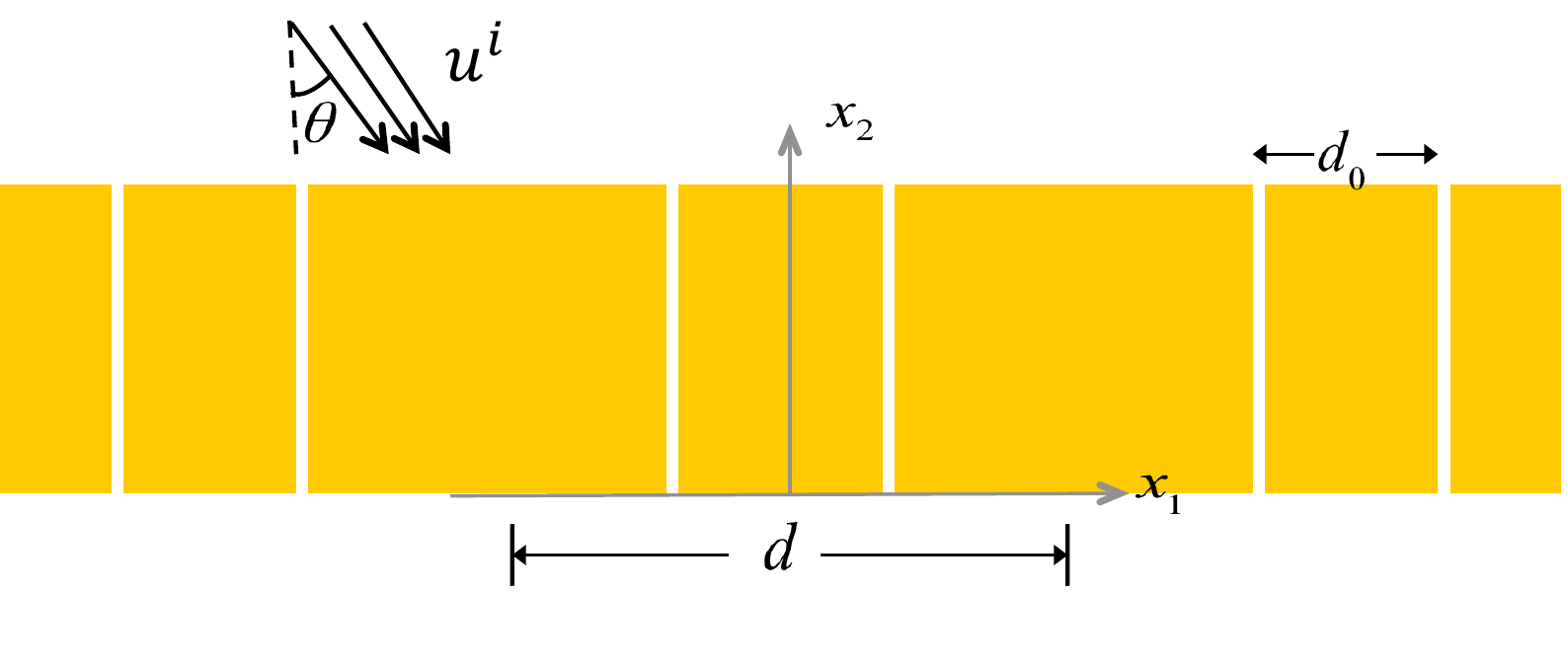}
\vspace*{-20pt}
\caption{Geometry of the periodic slit scattering problem on the $x_1x_2$ plane. Each period consists of two subwavelength slits $S_\varepsilon^{0,-}$ and $S_\varepsilon^{0,+}$, which
 have a rectangular shape of length $1$ and width $\varepsilon$.  The upper and lower aperture of the slit $S_{\varepsilon}^{0,\pm}$ is denoted as $\Gamma^{\pm}_{1,\varepsilon}$ and $\Gamma^{\pm}_{2,\varepsilon}$, respectively. The domain exterior to the perfect conductor is denoted as $\Omega_{\varepsilon}$, which consists of the slit region $S_\varepsilon$, the domains above the slab $\Omega_1$, and the domain below the slab $\Omega_2$.
 }\label{fig:prob_geo}
\end{center}
\end{figure}

\begin{figure}[!htbp]
\begin{center}
\includegraphics[height=3.5cm]{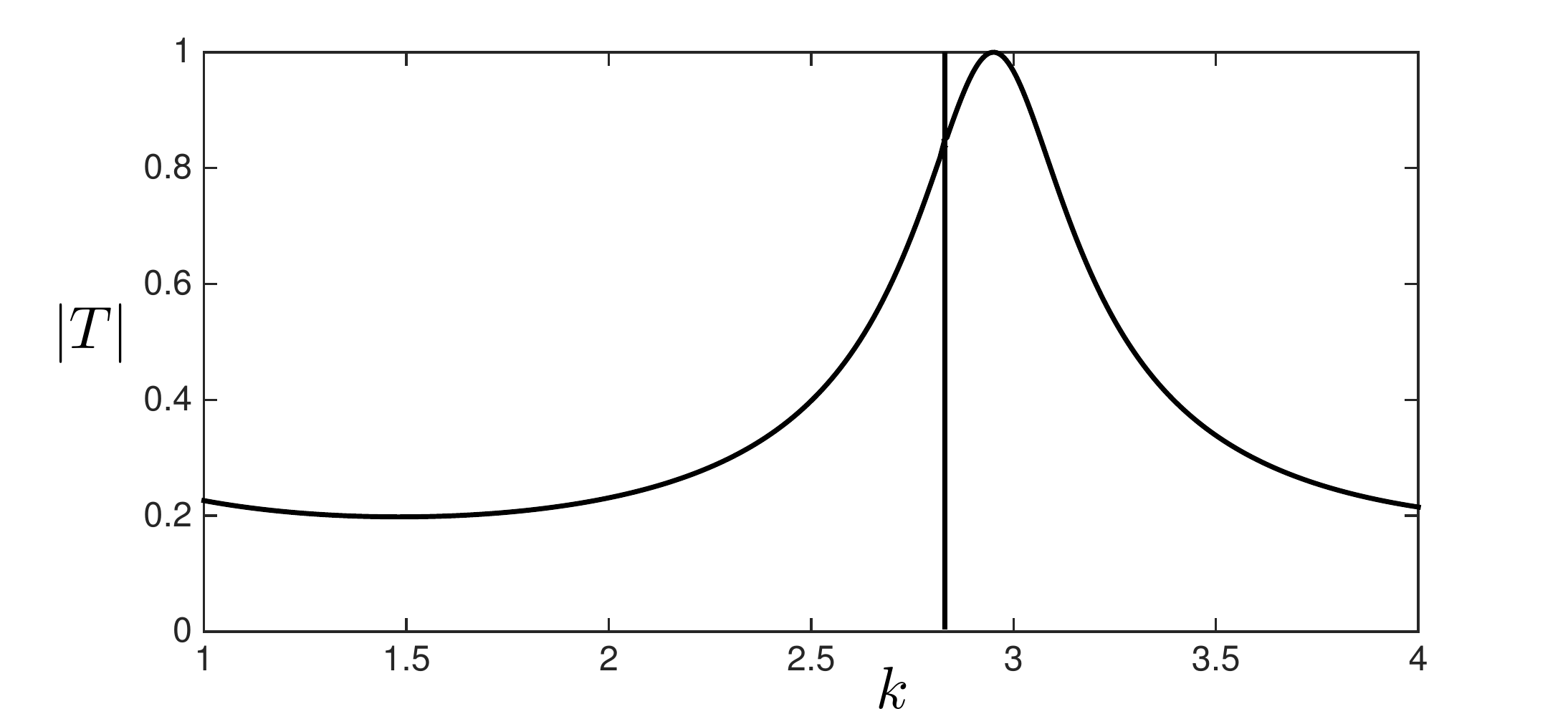}
\includegraphics[height=3.5cm, width=4cm]{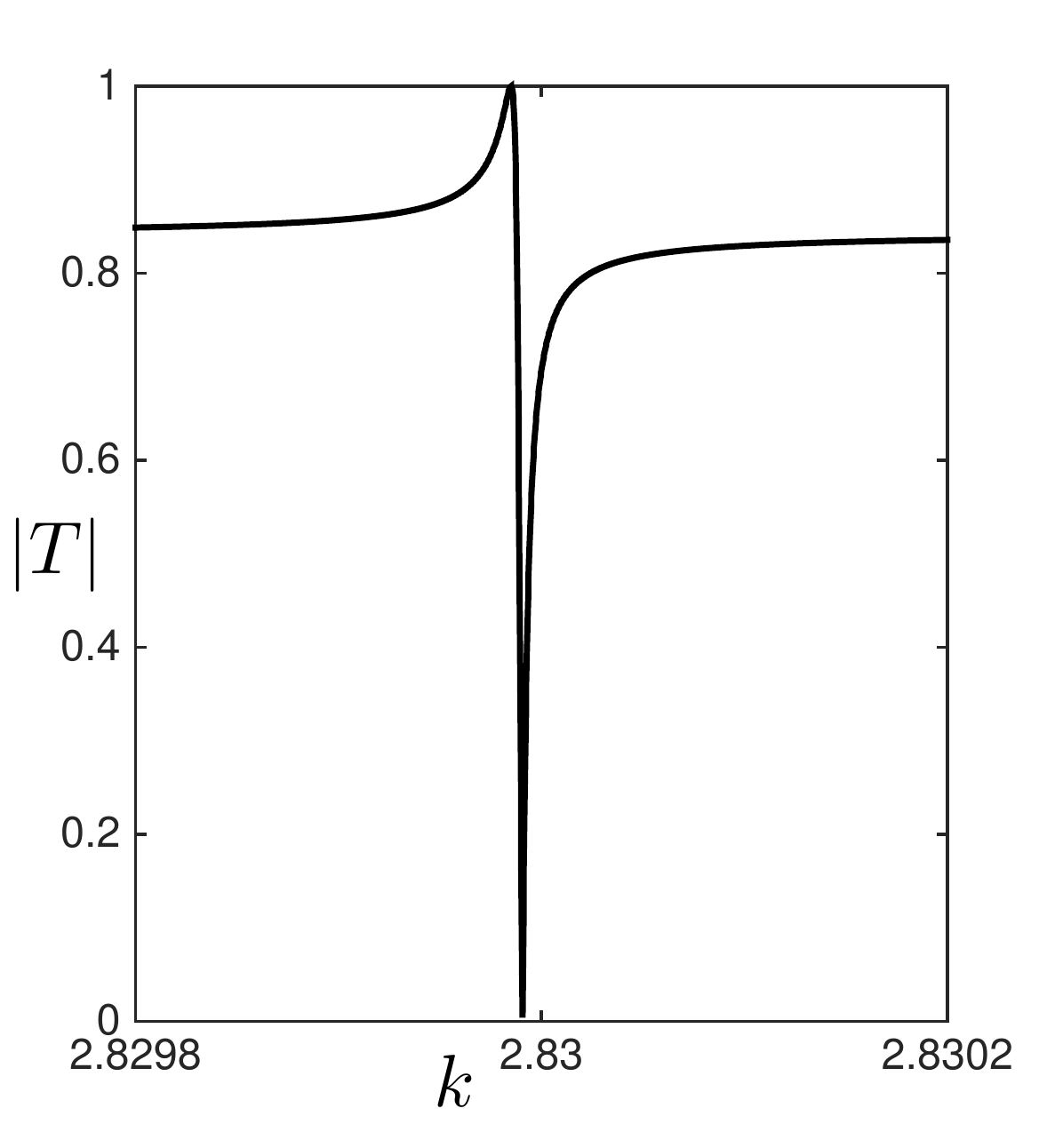}
\caption{Left: Transmission $|T|$ when $d=1$, $d_0=0.4$, $\varepsilon=0.05$, $\kappa=0.1$. An asymmetric line shape occurs near $k=2.83$. Right: Zoomed view of Fano resonance.
}\label{fig:transmission}
\end{center}
\end{figure}

In a series of studies \cite{lin15}--\cite{lin_zhang18_2}, we have established rigorous mathematical theories for the  EOT  and field enhancement in narrow slit structures perforated in a perfectly conducting metallic slab. Both a single slit structure and a periodic array of slit structures in various scaling regimes were considered.
The structures in the present work exhibit an infinite set of resonances similar to those, in which the imaginary part of the complex frequency is nonzero, resulting in Fabry-Perot type resonance.  But they also exhibit a finite set of real resonances, which are the eigenvalues mentioned above.  They occur when the Bloch wavenumber $\kappa$ is zero, and they move away from the real axis when $\kappa$ becomes nonzero.  We show that the field enhancement associated with these resonances is stronger than for the resonant frequencies that remain non-real.  This resonance is known as Fano resonance, and it is associated with sharp anomalies in the transmission of energy across the slab.  For the purely Fabry-Perot resonance, the amplification is on the order inversely proportional to the width $\varepsilon$ of the slits, uniformly in $\kappa$, whereas for the perturbed eigenvalues, the resonance is on the order of $1/(\kappa\varepsilon)$.

To be more specific, we investigate scattering by a metallic grating structure as shown in Figure~\ref{fig:prob_geo}, where each period consists of two identical narrow slits.
The symmetry effects a decoupling of even and odd modes about the centerline between the two slits when $\kappa\!=\!0$. 
An odd surface mode of the grating can exist at a frequency 
within the even continuum, and the frequency of this mode is the embedded eigenvalue
\cite{bonnet_starling94,shipman07,shipman10}.
It is not possible when there is only one slit per period because the slit can not support odd modes when it is too narrow.
If $\kappa$ is perturbed away from zero, the transmission coefficient across the slab exhibits an asymmetric shape as shown in Figure~\ref{fig:transmission}, with peak and dip at very close frequencies.

The perfectly conducting metallic slab occupies the domain $\{ x=(x_1,x_2); \, 0<x_2<1\}$ in the $x_1x_2$ plane.
The domain above and below the metallic slab are denoted by $\Omega_1$ and $\Omega_2$ respectively.
The slits, which are perforated in the metallic slab and invariant along the $x_3$ direction, 
occupy the region $\displaystyle{S_\varepsilon=\bigcup_{n=0}^{\infty} (S_\varepsilon^{(0)} + nd)}$, where $d$ is the size of the period and
$S_\varepsilon^{(0)}$ consists of two subwavelength slits $S_\varepsilon^{0,-}$ and $S_\varepsilon^{0,+}$.
The slits  $S_\varepsilon^{0,\pm}$ are given by
$$\textstyle S_\varepsilon^{0,\pm}:=\left\{(x_1,x_2)\;|\; \pm \frac{d_0}{2}-\frac{\varepsilon}{2} <x_1< \pm \frac{d_0}{2} + \frac{\varepsilon}{2},\; 0<x_2<1 \right\}, $$
where $d_0=O(d)$. Denote the upper and lower apertures of the slit $S_{\varepsilon}^{0,\pm}$ by $\Gamma^{\pm}_{1,\varepsilon}$ and $\Gamma^{\pm}_{2,\varepsilon}$ and exterior domain of the metallic structure by $\Omega_\varepsilon$.

We study the time-harmonic transverse magnetic situation, where the magnetic field is perpendicular to the $x_1x_2$ plane. The $x_3$ component of the incident field is
\begin{equation}
  u^\mathrm{inc}(x) = e^{i k( x_1\sin\theta \,-\, (x_2-1)\cos \theta )} = e^{i\zeta_0}e^{i\kappa x_1-i\zeta_0 x_2},
\end{equation}
in which $k$ is the free-space wavenumber, $\theta\in (-\pi/2,\pi/2)$ is the angle of incidence, $\kappa= k \sin \theta$ is the Bloch wavenumber, and $\zeta_0=\sqrt{k^2-\kappa^2}>0$.  The total field $u_\varepsilon(x)$ satisfies the following scattering problem:
\begin{eqnarray}\label{scatteringproblem}
  && \Delta u_\varepsilon + k^2 u_\varepsilon \;=\; 0 \quad \mbox{in} \; \Omega_\varepsilon, \label{eq:Helmholtz} \\
  && \frac{\partial u_{\varepsilon}}{\partial \nu} = 0  \quad \mbox{on} \; \partial \Omega_{\varepsilon}, \label{eq:Neumann}\\
  && u_\varepsilon(x_1+d,\,x_2) \;=\; e^{i\kappa d}u_\varepsilon(x_1,x_2), \label{eq:quasi-periodic}\\
  && u_\varepsilon(x_1,x_2) \;=\; u^\mathrm{inc}(x_1,x_2) + \sum_{n=-\infty}^{\infty} u_{n,1}^{s} e^{i \kappa_n x_1 + i\zeta_n  x_2 } \quad \mbox{in} \; \Omega_1, \label{eq:field1}\\
  && u_\varepsilon(x_1,x_2) \;=\; \sum_{n=-\infty}^{\infty} u_{n,2}^{s} e^{i \kappa_n x_1 - i\zeta_n  x_2 } \quad \mbox{in} \; \Omega_2. \label{eq:field2}
\end{eqnarray}
Equation (\ref{eq:Helmholtz}) is the Helmholtz partial differential equation, with $\Delta=\partial^2/\partial x_1^2 + \partial^2/\partial x_2^2$.
Equation (\ref{eq:Neumann}) is the Neumann boundary condition, with $\nu$ the unit normal vector pointing to $\Omega_\varepsilon$.
Equation (\ref{eq:quasi-periodic}) expresses the quasi-periodic property, which is consistent with the incident field.
The series in (\ref{eq:field1}) and (\ref{eq:field2}) are the Rayleigh-Bloch (Fourier) expansions for outgoing, or radiating, fields (cf.~\cite{bao95, bonnet_starling94, shipman10}), in which the coefficients $u_{n,i}^s$ are complex amplitudes.  
The constants $\kappa_n$ and $\zeta_n$ are defined by
\begin{equation}
  \kappa_n=\kappa+\textstyle\frac{2\pi n}{d} \quad \mbox{and} \quad
\zeta_n = \zeta_n(k,\kappa)= \sqrt{k^2-\kappa_n^2}\,,
\end{equation}
where the domain of the analytic square root function is taken to be $\mathbf{C}\backslash\{-it: t\geq 0\}$, with $\sqrt{1}=1$.  With this choice of square root,
\begin{equation}
 \zeta_n(k,\kappa)   = \left\{
\begin{array}{lll}
\vspace*{5pt}
\sqrt{k^2-\kappa_n^2}\,,  & \mbox{if}\;\abs{\kappa_n} \leq k, \\
i\sqrt{\kappa_n^2-k^2}\,,  & \mbox{if}\;\abs{\kappa_n} \geq k. \\
\end{array}
\right. 
\end{equation}
The Rayleigh modes with $\abs{\kappa_n} < k$ are propagating, and the modes with $\abs{\kappa_n} > k$ are evanescent.  The case of $\zeta_n=0$ is delicate, but we won't be concerned with it since we are interested in the regime in which $\zeta_n$ is real for $n=0$ and nonzero imaginary for $n\not=0$ (cf.~\eqref{diamond}).

By applying layer potential techniques and asymptotic analysis, we aim to
\begin{itemize}
\item [(i)] provide quantitative analysis of the embedded eigenvalues for the homogeneous scattering problem \eqref{eq:Helmholtz}--\eqref{eq:field2} when $\kappa=0$, and 
 their perturbations as resonances when $\kappa\neq0$ (Theorems \ref{thm:asym_res_eig} and  \ref{thm:asym_eig_perturb});
\item [(ii)] give a rigorous proof of Fano resonant transmission anomalies as shown in Figure~\ref{fig:transmission} for the periodic structure  (Theorem~\ref{thm:Fano}).
\item [(iii)] characterize the field amplification at Fano resonance (Theorem~\ref{thm:u_slit_res}).
\end{itemize}
The paper is organized as follows. We present an integral equation formulation for the scattering problem \eqref{eq:Helmholtz}--\eqref{eq:field2} in Section~\ref{sec:bie}.
In Section~\ref{sec:bie_asymptotic_analysis}, we derive the asymptotic expansions of the integral operators.
The asymptotic analysis of the embedded eigenvalues when $\kappa=0$ and their perturbations when $\kappa\neq0$ is given Section~\ref{sec:eig_res}.
The Fano resonance and the corresponding field enhancement is analyzed in Section~\ref{sec:fano_field_enhancement}.

\setcounter{equation}{0}
\section{Boundary-integral formulation}\label{sec:bie}

The scattering problem \eqref{eq:Helmholtz}--\eqref{eq:field2} can be formulated equivalently as a system of boundary-integral equations.  The development in this section is standard.

Due to the quasi-periodicity of the solution, one can restrict the Bloch wave number $\kappa$ to the first Brillouin zone $\kappa\in(-\pi/d,\pi/d]$.  Note that our incident field is the propagating harmonic for $n\!=\!0$.
For each fixed $\kappa\in(-\pi/d,\pi/d]$, let $g(x, y)= g(k,\kappa; x,y)$ be the quasi-periodic Green function, which satisfies the equation
$$ \Delta g(x,y) + k^2 g(x,y)=e^{i\kappa(x_1-y_1)} \sum_{n=-\infty}^{\infty} \delta(x_1-y_1-nd)\delta(x_2-y_2),$$
with $x=(x_1,x_2)$ and $y=(y_1,y_2)$ in $\mathbf{R}^2$.
Its Rayleigh-Bloch expansion (cf.~\cite{linton98}) is
\begin{equation}\label{eq:gd}
g(k,\kappa; x,y) = -\frac{i}{2d}  \sum_{n=-\infty}^{\infty} \frac{1}{\zeta_n(k,\kappa)} e^{i \kappa_n(x_1-y_1)+i\zeta_n |x_2-y_2| }.
\end{equation}
The exterior Green functions for the domains $\Omega_1$ and $\Omega_2$ above and below the slab
 with the Neumann boundary condition $\partial g^\mathrm{e}(\kappa; x,y)/\partial \nu_y=0$ along the boundaries $\{y_2=1\}$ and $\{y_2=0\}$ are equal to $g^\mathrm{e}(x,y)=g^\mathrm{e}(k,\kappa; x,y) =g(k,\kappa; x,y)+g(k,\kappa; x',y)$, where 
\begin{equation*}
x' = \left\{
\begin{array}{ll}
(x_1, 2-x_2) & \mbox{if} \; x, y \in \Omega_1,  \\
(x_1,-x_2) &  \mbox{if} \; x, y \in \Omega_2.
\end{array}
\right.
\end{equation*}

The interior Green functions $g_\varepsilon^{\mathrm{i},\pm}(x,y)$ in the slits $S_{\varepsilon}^{0,\pm}$ with the Neumann boundary condition are
\begin{equation*}
  \textstyle g_\varepsilon^{\mathrm{i},\pm}(k;x,y): = g_\varepsilon^{\mathrm{i},0}(k;x_1 \mp\frac{d_0}{2}, x_2; y_1 \mp\frac{d_0}{2}, y_2), 
\end{equation*}
in which $g_\varepsilon^{\mathrm{i},0}(k;x,y)$ satisfies
$$ \textstyle\Delta g_\varepsilon^{\mathrm{i},0}(k;x,y) + k^2 g_\varepsilon^{\mathrm{i},0}(k;x,y) = \delta(x-y), \quad x, y\in (-\frac{\varepsilon}{2},\frac{\varepsilon}{2})\times(0,1), $$
and it can be expressed explicitly as
$$ g_\varepsilon^{\mathrm{i},0}(k;x,y)= \sum_{m,n=0}^\infty c_{mn}\phi_{mn}(x)\phi_{mn}(y),$$
with $c_{mn}\!=\![k^2\!-\!(m\pi/\varepsilon)^2\!-\!(n\pi)^2]^{-1}$,
$ \phi_{mn}(x)\!=\!\sqrt{\frac{a_{mn}}{\varepsilon}}\cos\left(\frac{m\pi}{\varepsilon}\left(x_1+\frac{\varepsilon}{2}\right)\right) \cos(n\pi x_2)$,~and
\begin{equation*}
a_{mn} = \left\{
\begin{array}{llll}
1  & m=n=0, \\
2  & m=0, n\ge 1 \quad \mbox{or} \quad n=0, m\ge 1, \\
4  & m\ge 1, n \ge 1.
\end{array}
\right.
\end{equation*}

Applying Green's theorem in the reference period $\Omega^{(0)}:=\{ x \in \Omega_\varepsilon \; | \;   -\frac{d}{2}<x_1< \frac{d}{2} \}$ yields the following lemma for the total field $u_\varepsilon$.
\begin{lemma}\label{eq:u_eps_formula}
Let $u_\varepsilon(x)$ be the solution of the scattering problem \eqref{eq:Helmholtz}--\eqref{eq:field2}, then
\begin{eqnarray*}
u_\varepsilon(x) &=& \int_{\Gamma^+_{1,\varepsilon} \cup \Gamma^-_{1,\varepsilon}} g^\mathrm{e}(x,y) \frac{\partial u_\varepsilon(y)}{\partial y_2} ds_y + u^\mathrm{inc}(x)+ u^\mathrm{refl}(x)  \quad \mbox{for} \;\; x\in\Omega^{(0)} \cap \Omega_1, \\
u_\varepsilon(x) &=& -\int_{\Gamma^+_{2,\varepsilon} \cup \Gamma^-_{2,\varepsilon}} g^\mathrm{e}(x,y) \frac{\partial u_\varepsilon(y)}{\partial y_2} ds_y  \quad \mbox{for} \;\; x\in\Omega^{(0)} \cap \Omega_2, \\
u_\varepsilon(x) &=& -\int_{\Gamma^-_{1,\varepsilon}} g_\varepsilon^{\mathrm{i},-}(x,y) \frac{\partial u_\varepsilon(y)}{\partial y_2} ds_y 
+\int_{\Gamma^-_{2,\varepsilon} } g_\varepsilon^{\mathrm{i},-}(x,y) \frac{\partial u_\varepsilon(y)}{\partial y_2} ds_y \quad \mbox{for} \;\; x\in S_\varepsilon^{0,-}, \\
u_\varepsilon(x) &=& -\int_{\Gamma^+_{1,\varepsilon} } g_\varepsilon^{\mathrm{i},+}(x,y) \frac{\partial u_\varepsilon(y)}{\partial y_2} ds_y + \int_{\Gamma^+_{2,\varepsilon} } g_\varepsilon^{\mathrm{i},+}(x,y) \frac{\partial u_\varepsilon(y)}{\partial y_2} ds_y \quad \mbox{for} \;\; x\in S_\varepsilon^{0,+}.
\end{eqnarray*}
Here, $u^\mathrm{refl}(x)=e^{i (\kappa x_1 + \zeta_0 (x_2-1))}$ is the reflected field of the ground plane $\{x_2=1\}$ without the slits.
\end{lemma}

Taking the limit of layer potentials in this lemma to the slit apertures and imposing the continuity condition leads to the following system of four integral equations:
\begin{equation}\label{eq:scattering2}
\left\{
\begin{array}{l}
\displaystyle{\int_{\Gamma^+_{1,\varepsilon} \cup \Gamma^-_{1,\varepsilon}}  g^\mathrm{e}(x,y)\dfrac{\partial u_\varepsilon(y)}{\partial y_2} ds_y 
+\int_{\Gamma^\mp_{1,\varepsilon}} g_\varepsilon^{\mathrm{i},\mp}(x,y) \dfrac{\partial u_\varepsilon(y)}{\partial y_2} ds_y
-\int_{\Gamma^\mp_{2,\varepsilon}} g_\varepsilon^{\mathrm{i},\mp}(x,y) \dfrac{\partial u_\varepsilon(y)}{\partial y_2} ds_y}\;+ \\ \\
 \hspace{1em} +\; u^\mathrm{inc}(x)+u^\mathrm{refl}(x) =0, \quad \mbox{for} \,\, x\in\Gamma^\mp_{1,\varepsilon}, \\ \\ 
\displaystyle{\int_{\Gamma^+_{2,\varepsilon} \cup \Gamma^-_{2,\varepsilon}}  g^\mathrm{e}(x,y)\dfrac{\partial u_\varepsilon(y)}{\partial y_2} ds_y
-\int_{\Gamma^\mp_{1,\varepsilon}} g_\varepsilon^{\mathrm{i},\mp}(x,y) \dfrac{\partial u_\varepsilon(y)}{\partial y_2} ds_y
+\int_{\Gamma^\mp_{2,\varepsilon}} g_\varepsilon^{\mathrm{i},\mp}(x,y) \dfrac{\partial u_\varepsilon(y)}{\partial y_2} ds_y
=0}, \\ \\
 \hspace{1em}\; \mbox{for} \,\, x\in\Gamma^\mp_{2,\varepsilon}.
\end{array}
\right.
\end{equation}

The slit apertures are rescaled to the $\varepsilon$-independent variable $X\in I : = (-\frac{1}{2},\frac{1}{2})$ by
$$ x_1 =  \varepsilon X \pm \frac{d_0}{2}  \quad \mbox{for} \; (x_1,1) \in \Gamma_{1,\varepsilon}^\pm \;\; \mbox{and} \; (x_1,0) \in \Gamma_{2,\varepsilon}^\pm.$$
The following quantities will be used in the boundary-integral formulation of the scattering problem,
\begin{eqnarray*}
&& \varphi_1^\pm(X):= \frac{\partial u_\varepsilon}{\partial y_2}( \varepsilon X\pm\textstyle \frac{d_0}{2},\, 1), 
 \quad \varphi_2^\pm(X):= -\frac{\partial u_\varepsilon}{\partial y_2}( \varepsilon X\pm\textstyle \frac{d_0}{2},\, 0),
   \\ 
&& f^\pm(X):= -\frac{1}{2} (u^\mathrm{inc}+u^\mathrm{refl})(\varepsilon X\pm\textstyle \frac{d_0}{2},\, 1) = -e^{i \kappa   (\varepsilon X\pm\frac{d_0}{2} )}.
\end{eqnarray*}
The Green functions are also rescaled:
\begin{eqnarray*}
 G_\varepsilon^\mathrm{e}(X, Y) &:=&  \textstyle g^\mathrm{e}( \varepsilon X\pm\frac{d_0}{2},1;  \varepsilon Y\pm\frac{d_0}{2},1) = 
g^\mathrm{e}( \varepsilon X\pm\frac{d_0}{2},0;  \varepsilon Y\pm\frac{d_0}{2},0) \\
 && = -\frac{i}{d} \sum_{n=-\infty}^{\infty} \frac{1}{\zeta_n(k,\kappa)   } e^{i \kappa_n\varepsilon(X-Y)};  
 \end{eqnarray*}
 \vspace*{-20pt}
 \begin{eqnarray*}
 G_\varepsilon^{\mathrm{e},\pm}(X, Y) &:=& \textstyle g^\mathrm{e}( \varepsilon X\pm\frac{d_0}{2},1; \varepsilon Y \mp \frac{d_0}{2},1) =
  \textstyle g^\mathrm{e}( \varepsilon X\pm\frac{d_0}{2},0; \varepsilon Y \mp \frac{d_0}{2},0) \\
 && =  -\frac{i}{d} \sum_{n=-\infty}^{\infty} \frac{1}{\zeta_n(k,\kappa)   } e^{i \kappa_n (\varepsilon(X-Y)\pm d_0)};
 \end{eqnarray*}
  \vspace*{-20pt}
 \begin{eqnarray*}
 G_\varepsilon^\mathrm{i}(X, Y) &:=& \textstyle g_\varepsilon^{\mathrm{i},\pm}( \varepsilon X\pm\frac{d_0}{2}, 1; \varepsilon Y\pm\frac{d_0}{2}, 1 ) = 
\textstyle g_\varepsilon^{\mathrm{i},\pm}( \varepsilon X\pm\frac{d_0}{2}, 0; \varepsilon Y\pm\frac{d_0}{2}, 0 )  \\
 &=& \varepsilon^{-1} \sum_{m,n=0}^\infty c_{mn}a_{mn}\cos\left(m\pi (X+\textstyle\frac{1}{2})\right ) \cos\left(m\pi (Y+\textstyle\frac{1}{2})\right ); \\ 
\tilde  G_\varepsilon^\mathrm{i}(X, Y) &:=& \textstyle g_\varepsilon^{\mathrm{i},\pm}(\varepsilon X\pm\frac{d_0}{2}, 1;  \varepsilon Y\pm\frac{d_0}{2}, 0 )
=  g_\varepsilon^{\mathrm{i},0}(\varepsilon X, 1; \varepsilon Y, 0) \\
 && = \varepsilon^{-1} \sum_{m,n=0}^\infty(-1)^n c_{mn}a_{mn} \cos\left(m\pi (X+\textstyle\frac{1}{2})\right ) \cos\left(m\pi (Y+\textstyle\frac{1}{2})\right ).
\end{eqnarray*}
Define the following rescaled boundary-integral operators for $X\in I$:
\begin{eqnarray*}
 && [T^\mathrm{e} \varphi](X) = \int_I G_\varepsilon^\mathrm{e}(X, Y)  \varphi(Y) dY, \quad [T^{\mathrm{e},\pm} \varphi](X) = \int_I G_\varepsilon^{\mathrm{e},\pm}(X, Y)  \varphi(Y) dY;   
 \\
 && [T^\mathrm{i}  \varphi] (X) = \int_I  G_\varepsilon^\mathrm{i}(X, Y) \varphi(Y) dY,  \quad [\tilde T^\mathrm{i}  \varphi] (X) = \int_I  \tilde G_\varepsilon^\mathrm{i}(X, Y) \varphi(Y) dY. 
\end{eqnarray*}
Then $T^\mathrm{e}$ and $T^\mathrm{i}$ are bounded operators from $H^{-1/2}(I)$ to $H^{1/2}(I)$, and $T^{\mathrm{e},\pm}$ and $\tilde T^\mathrm{i}$
are smooth operators due to the smoothness of the kernels.

\begin{proposition}
The system (\ref{eq:scattering2}) is equivalent to the system
$\mathbb{T}\boldsymbol{\varphi}=\varepsilon^{-1}\mathbf{f}$, with
\begin{equation}\label{eq:optT}
\mathbb{T}=\left[
\begin{array}{cccc}
T^\mathrm{e}+T^\mathrm{i}    &  T^{\mathrm{e},-}   & \tilde T^\mathrm{i} & 0  \\
 T^{\mathrm{e},+}     &  T^\mathrm{e}+T^\mathrm{i} & 0 & \tilde T^\mathrm{i}     \\
  \tilde T^\mathrm{i}  & 0  & T^\mathrm{e}+T^\mathrm{i}    &  T^{\mathrm{e},-}  \\
   0  & \tilde T^\mathrm{i}  & T^{\mathrm{e},+}  & T^\mathrm{e}+T^\mathrm{i}   
\end{array}
\right], \quad
\boldsymbol{\varphi} =
\left[
\begin{array}{cccc}
\varphi_1^-    \\
\varphi_1^+   \\
\varphi_2^-   \\
\varphi_2^+
\end{array}
\right], \quad
\mathbf{f}=\left[
\begin{array}{cccc}
2f^-  \\
2f^+  \\
0   \\
0
\end{array}
\right]
\end{equation}
and $f^\pm(X)= -e^{i\kappa(\varepsilon X\pm\, d_0/2)}$.
\end{proposition}

\section{Asymptotic expansion of the integral operators}\label{sec:bie_asymptotic_analysis}

The rest of this paper will concern parameters $k$ and $\kappa$ such that
$|\kappa|<\pi/d$ and $(\kappa+2n\pi/d)^2>k^2$ for all integers $n\not=0$.  In particular, $0<k<2\pi/d$.
This is the parameter regime in which $\zeta_0$ is real and $\zeta_n$ is imaginary for $n\not=0$, that is, there is exactly one propagating Rayleigh mode.  In the $\kappa$-$k$ plane, it is a diamond-shaped region $D_1$,
\begin{equation}\label{diamond}
  D_1 = \left\{ (\kappa,k) :  |\kappa|<\pi/d,\; 0<k<\left| \kappa+2n\pi/d \right| \;\forall n\not=0  \right\}.
\end{equation}

The behavior of the rescaled Green functions to leading order in $\varepsilon$ and $|\kappa|$ is independent of $(X,Y)$.  
This will reduce the leading asymptotics of (\ref{eq:optT}) to a four-dimensional system involving the average field values on the ends of the slits $\Gamma^\pm_{i,\varepsilon}$.  The following quantities, which depend on $k$, $\kappa$, and $\epsilon$, describe this behavior.

\begin{eqnarray}
 && \beta_\mathrm{e}(k,\kappa, \varepsilon)= \frac{1}{\pi} \ln \frac{2\pi\,\varepsilon}{d} + \sum_{n\neq 0} \left( \frac{1}{2\pi} \frac{1}{|n|} - \frac{i}{d}  \frac{1}{\zeta_n(k,\kappa)}\right)
  - \frac{i}{d}  \frac{1}{\zeta_0(k,\kappa)}, \label{beta_e} \\
&& \beta^\pm(k,\kappa) = -\frac{i}{d} \sum_{n=-\infty}^{\infty} \frac{1}{\zeta_n(k,\kappa)   } e^{\pm i \kappa_n d_0},
\label{beta_pm}  \\
&& \beta_\mathrm{i}(k, \varepsilon)= \frac{1}{\varepsilon\, k\tan k } +  \frac{2\ln 2}{\pi},  \quad\quad \tilde \beta(k,\varepsilon) = \frac{1}{ \varepsilon\,k\sin k},  \label{beta_i_tilde} \\
&& \beta(k,\kappa, \varepsilon)=\beta_\mathrm{e}(k,\kappa, \varepsilon) + \beta_\mathrm{i}(k, \varepsilon), \quad 
\gamma(k, \kappa) = \beta(k,\kappa,\varepsilon) -  \frac{1}{\varepsilon\,k\tan k } - \frac{1}{\pi} \ln \varepsilon. \label{gamma}  
\end{eqnarray}

The asymptotic expansions for the kernels $G_\varepsilon^\mathrm{e}$,  $G_\varepsilon^{\mathrm{e},\pm}$, $G^i$ and $\tilde G^i$ are given in the following two lemmas.
\begin{lemma} \label{lem:green_ext_asy} 
For $|\kappa| \ll 1$ and $\varepsilon \ll 1$, if $k \in (0, 2\pi/d)$, then
\begin{eqnarray}\label{eq:Ge_asy}
G_\varepsilon^\mathrm{e}(X, Y)&=&\beta_\mathrm{e}(k, \kappa,\varepsilon) + \frac{1}{\pi}  \ln |X-Y| +  r_\mathrm{e}(\kappa,\varepsilon; X,Y), \label{Ge_exp} \\
G_\varepsilon^{\mathrm{e},\pm}(X, Y)&=& \beta^\pm(k, \kappa)  +  \rho^\pm(\kappa,\varepsilon; X,Y).       \label{Ge_pm_asy}  
\end{eqnarray}
Here $r_\mathrm{e}$, $\rho^\pm$ are bounded functions with $r_\mathrm{e} \sim O(\varepsilon^2|\ln\varepsilon|+|\kappa|\varepsilon)$ and $\rho^\pm\sim O(\varepsilon)$ for all $X, Y\in I$.
In addition, the following holds:
\begin{enumerate}
\item[(i)] For $\kappa=0$, there holds
\begin{eqnarray} 
\hspace{15pt} && \beta_\mathrm{e}^\pm(k, 0) = \hat\beta_\mathrm{e}(k) :=  -\frac{i}{d} \cdot \frac{1}{\zeta_0(k)}  - \sum_{n=1}^{\infty} \frac{2}{\sqrt{(2\pi n)^2-(kd)^2}} \cos(\kappa_n d_0), \label{eq:beta_kappa0} \\
\hspace{15pt} && r_\mathrm{e}(0, \varepsilon; X,Y) = \hat r_\mathrm{e}(|X-Y|), \quad \rho^\pm(0,\varepsilon; X,Y) = \hat\rho( |\pm d_0+\varepsilon(X-Y)|)  \label{eq:re_kappa0}
\end{eqnarray}
for some real-valued functions $\hat r_\mathrm{e}$ and $\hat \rho$, where  $\hat r_\mathrm{e}\sim O(\varepsilon^2|\ln\varepsilon|)$ and $\hat \rho \sim O(\varepsilon)$.
\item[(ii)] If $|\kappa| \ll 1$, then
\begin{equation*}
 r_\mathrm{e}(\kappa,\varepsilon; X,Y) = r_\mathrm{e}(0,\varepsilon; X,Y)  + O(\kappa\varepsilon), \quad \rho^\pm(\kappa,\varepsilon; X,Y) = \rho^\pm(0,\varepsilon; X,Y) + O(\kappa\varepsilon). 
\end{equation*}
\end{enumerate}
\end{lemma}

\begin{proof}  We first derive the asymptotic expansion of $ G^\mathrm{e}_\varepsilon(X,Y)$.  From the definition of the Green's function, we see that
\begin{equation}\label{eq:sum_Ge1}
  G^\mathrm{e}_\varepsilon(X,Y) \;=\;
  \frac{1}{d} \sum_{n\in\mathbb{Z}} \frac{1}{i\zeta_n(k,\kappa)}\, e^{i\kappa_n\varepsilon(X-Y)}.
\end{equation}
Set $a=2\pi/d$, and $Z=X-Y$.
For $n\neq0$, from the definition of $\zeta_n$, one can write $i\zeta_n = -\sqrt{(\kappa+an)^2-k^2\,}$ and it follows that
\begin{equation}
  \frac{1}{i\zeta_n(k,\kappa)} \;=\; -\frac{1}{a|n|}
  \left(\sqrt{1+\frac{2\kappa}{an}+\frac{\kappa^2-k^2}{(an)^2}\,}\right)^{-1}.
\end{equation}

For $(\kappa,k)\in D_1$, we have $\left| \frac{2\kappa}{an}+\frac{\kappa^2-k^2}{(an)^2} \right|  < 1$ for $n\neq0$.
Applying the Taylor expansion and splitting $(i\zeta_n)^{-1}$ into its even and odd parts with respect to the variable $n$ yields
$\frac{1}{i\zeta_n(k,\kappa)} \;=\; h^\mathrm{e}_n(k,\kappa) + h^\mathrm{o}_n(k,\kappa)$,
and consequently,
\begin{equation}\label{eq:sum_Ge2}
  \sum_{n\in\mathbb{Z}} \frac{e^{i\kappa_n\varepsilon Z}}{i\zeta_n(k,\kappa)}\, 
  \;=\;
  e^{i\kappa\varepsilon Z}
  \left( \frac{1}{i\zeta_0} 
  + 2 \sum_{n=1}^\infty h^\mathrm{e}_n(k,\kappa) \cos(n a\varepsilon Z)
  + 2i \sum_{n=1}^\infty h^\mathrm{o}_n(k,\kappa) \sin(n a \varepsilon Z)
  \right).
\end{equation}
If $\kappa= 0$, it can be calculated that
$$ h^\mathrm{e}_n(k,0)  = - \sum_{n = 1}^\infty \frac{1}{an} \left( 1+\sum_{m=1}^\infty \frac{c_m}{n^{2m}} \right),  \quad \mbox{where} \quad  c_m=\frac{1\cdot3\cdots (2m-1)}{2^m m!}\left(\frac{k}{a }\right)^{2m}, $$
and $h^\mathrm{o}_n(k,0)=0$, thus
$$ 2 \sum_{n=1}^\infty h^\mathrm{e}_n(k,0) \cos(na\varepsilon Z) =  -\sum_{n = 1}^\infty \frac{2}{an} \cos(na\varepsilon Z)  - \frac{1}{a}  \sum_{m=1}^\infty c_m \sum_{n=1}^\infty  \frac{1}{n^{2m+1}}\cos(na\varepsilon Z). $$
Using the relations (cf.~\cite{collin, kress})
\begin{eqnarray*}
  && - \sum_{n = 1}^\infty \frac{2}{n} \cos \big( n a\varepsilon Z\big) = \ln\left(4\sin^2\frac{a\varepsilon Z}{2}\right), \label{eq:exp_cos1} \\
  && \sum_{n = 1}^\infty \frac{1}{n^{2m+1}} \cos \left(na\varepsilon Z\right) = \sum_{n = 1}^\infty \frac{1}{n^{2m+1}}+ O(\varepsilon^{2m}Z^{2m}) \cdot |\ln (\varepsilon|Z|) |  \quad (m\ge 1),
\end{eqnarray*}
in which the big-O remainder is uniform over $m$, we obtain
\begin{eqnarray}
  2\sum_{n=1}^\infty h^\mathrm{e}_n(k,\kappa) \cos(an\varepsilon Z)
   &=& \frac{1}{a} \ln\left( 4\sin^2\frac{a\varepsilon Z}{2} \right) 
    + 2 \sum_{n=1}^\infty \left( h^\mathrm{e}_n(k,0) + \frac{1}{an} \right) 
    + O(\varepsilon^2|\ln\varepsilon|) \nonumber  \\
   &=& 
 \label{eq:sum_exp} 
 \frac{2}{a} \ln (a\varepsilon|Z|) + 2\sum_{n=1}^\infty \left(  \frac{1}{i\zeta_n(k,0)} + \frac{1}{an} \right) + O(\varepsilon^2|\ln\varepsilon|). 
\end{eqnarray}
The desired asymptotic expansion for $G_\varepsilon^\mathrm{e}(X, Y)$ follows by combining \eqref{eq:sum_Ge1}, \eqref{eq:sum_Ge2} and \eqref{eq:sum_exp}. In addition, the above expansion shows that $r_\mathrm{e}(0,\varepsilon;X, Y)$ is a function of $|Z|:=|X-Y|$ and is real when $\kappa=0$.

If $|\kappa|\ll 1$, the even part can be written as $\displaystyle{h^\mathrm{e}_n(k,\kappa) \;=\; -\frac{1}{a|n|} +
    \frac{1}{|an|^3} \sum_{m=0}^\infty \frac{c_m(k,\kappa)}{(an)^{2m}}}$.
Similar calculations yield
\begin{equation}\label{eq:sum_even}
2\sum_{n=1}^\infty h^\mathrm{e}_n(k,\kappa) \cos(na\varepsilon Z)
  = \frac{2}{a} \ln (a\varepsilon|Z|) 
       + 2\sum_{n=1}^\infty \left( h^\mathrm{e}_n(k,\kappa) + \frac{1}{an} \right) + O(\varepsilon^2|\ln\varepsilon|).
\end{equation}
The odd term $h^\mathrm{o}_n(k,\kappa)$ has the form
$ \displaystyle{
  h^\mathrm{o}_n(k,\kappa) \;=\;
    \kappa\,\mathrm{sgn}(n) \sum_{m=1}^\infty \frac{b_m(k,\kappa)}{(an)^{2m}} }, $
and thus
\begin{equation*}
 \sum_{n=1}^\infty h^\mathrm{o}_n(k,\kappa) \sin(na\varepsilon Z) =   \kappa\, \sum_{m=1}^\infty b_m(k,\kappa)
  \sum_{n=1}^\infty \frac{1}{(an)^{2m}} \sin(na\varepsilon Z) = :  \kappa\, \sum_{m=1}^\infty b_m(k,\kappa) A_m(z).
\end{equation*}
By noting that $b_1=1$, we obtain (cf.~\cite{collin})
$$ A_1(z) =  \sum_{n=1}^\infty \frac{1}{(na)^{2}} \sin(na\varepsilon Z) = - \frac{1}{a} \varepsilon Z \ln (a \varepsilon |Z|) + O(\varepsilon). $$
On other hand, for $m>1$, $ A_m = O(\varepsilon)$. Hence,
\begin{equation}\label{eq:sum_odd}
 2 \sum_{n=1}^\infty h^\mathrm{o}_n(k,\kappa) \sin(an\varepsilon Z)  =  - \frac{2}{a} \kappa\varepsilon Z \ln (a \varepsilon |Z|) + O(\kappa\varepsilon).
\end{equation}

Substituting the sums \eqref{eq:sum_even} and \eqref{eq:sum_odd} into \eqref{eq:sum_Ge2} and using the expansion 
$e^{i\kappa\varepsilon Z}= 1+i \kappa\varepsilon Z + O((\kappa\varepsilon)^2)$,
we obtain
\begin{eqnarray*}
\sum_{n\in\mathbb{Z}} \frac{e^{i\kappa_n\varepsilon Z}}{i\zeta_n(k,\kappa)} &=&  \frac{e^{i\kappa\varepsilon Z}}{i\zeta_0(k,\kappa)} + \frac{2}{a} \ln (a\varepsilon|Z|) + 2 \sum_{n=1}^\infty \left( h^\mathrm{e}_n(k,\kappa) + \frac{1}{an} \right) 
    + O(\varepsilon^2|\ln\varepsilon|) + O(\kappa\varepsilon) \\
  &=& \frac{2}{a} \ln (a\varepsilon|Z|) +  \sum_{n\neq0} \left(  \frac{1}{i\zeta_n(k,\kappa)}  + \frac{1}{a|n|} \right) +  \frac{1}{i\zeta_0(k,\kappa)}
    + O(\varepsilon^2|\ln\varepsilon|) + O(\kappa\varepsilon).
\end{eqnarray*}
Therefore, the desired expansion
$ G_\varepsilon^\mathrm{e}(X, Y)=\beta_\mathrm{e}(k, \kappa,\varepsilon) + \frac{1}{\pi}  \ln |X-Y| +  r_\mathrm{e}(\kappa,\varepsilon; X,Y) $
follows, where $\beta_e$ is defined in \eqref{beta_e} and $r_\mathrm{e}=O(\varepsilon)$. In addition, from the above calculations, it is clear that
$r_\mathrm{e}(\kappa,\varepsilon; X,Y) = r_\mathrm{e}(0,\varepsilon; X,Y) + O(\kappa\varepsilon) $.

Now for 
$\displaystyle{G_\varepsilon^{\mathrm{e},\pm}(X, Y)  =  -\frac{i}{d} \sum_{n=-\infty}^{\infty} \frac{1}{\zeta_n(k,\kappa)   } e^{i \kappa_n (\varepsilon(X-Y)\pm d_0)}}$,
an analogous expansion as $G_\varepsilon^e(X,Y)$ leads to $G_\varepsilon^{\mathrm{e},\pm}(X, Y)   = \beta^\pm(k, \kappa)  +  \rho^\pm(\kappa,\varepsilon; X,Y)$,
where $\beta^\pm$ is defined in \eqref{beta_pm} and $\rho^\pm = O(\varepsilon)$.
In particular, when $\kappa=0$, it follows that
$$ \beta^\pm(k, 0) =\hat\beta(k) :=  -\frac{i}{d} \cdot \frac{1}{\zeta_0(k)}  - \sum_{n=1}^{\infty} \frac{2}{\sqrt{(2\pi n)^2-(kd)^2}} \cos(\kappa_n d_0), $$  
and the high-order terms take the form of
$$  \rho^\pm(0,\varepsilon; X,Y) =  \sum_{n=1}^{\infty} \frac{2}{\sqrt{(2\pi n)^2-(kd)^2}} \Big( \cos(\kappa_n d_0) - \cos(\kappa_n (\pm d_0+\varepsilon(X-Y)) \Big). $$
From the above expressions, it is seen that $\rho^\pm(0,\varepsilon; X,Y) = \hat \rho( |\pm d_0+\varepsilon(X-Y)|)$ for some real-valued function $\hat\rho$.   
\end{proof}

\medskip

\begin{rem}\label{rmk:green_ext_asy}
Asymptotic expansions can be obtained for $(\kappa,k) \not\in D_1$~\cite{lin_zhang18_1}.  In this case, $\hat r_\mathrm{e}$ and $\hat\rho$ in \eqref{eq:beta_kappa0} are no longer real valued.
\end{rem}

\medskip

\begin{lemma} \label{lem:green_int_asy} 
If $k\varepsilon \ll 1$, then
\begin{eqnarray}\label{eq:Gi_asy}
\hspace*{25pt} 
G_\varepsilon^\mathrm{i}(X, Y) 
 &=& \beta_\mathrm{i}(k ,\varepsilon) + \frac{1}{\pi} \left[ \ln \abs{\sin \left(\frac{\pi(X-Y)}{2}\right)} + \ln \abs{\sin \left(\frac{\pi(X+Y+1)}{2}\right)} \right] \nonumber\\
  && +\; r_{\mathrm{i},1}(\varepsilon; |X-Y|)  +  r_{\mathrm{i},2}(\varepsilon; |X+Y+1|).     \label{Gi_exp}  \\
\hspace*{25pt}
\tilde G_\varepsilon^\mathrm{i}(X, Y) &= &  \tilde \beta(k,\varepsilon) +  \tilde r_{\mathrm{i},1}(\varepsilon; |X-Y|) + \tilde r_{\mathrm{i},2}(\varepsilon; |X+Y+1|).
 \label{tGi_exp} 
\end{eqnarray}
Here $r_{\mathrm{i},1}$, $r_{\mathrm{i},2}$, $\tilde r_{\mathrm{i},1}$, and $\tilde r_{\mathrm{i},2}$ are bounded and real functions with
$r_{\mathrm{i},1}\sim O(\varepsilon^2)$, $r_{\mathrm{i},2} \sim O(\varepsilon^2)$, $\tilde r_{\mathrm{i},1} \sim O(e^{-1/\varepsilon})$, 
and $\tilde r_{\mathrm{i},2} \sim O(e^{-1/\varepsilon})$ for all $X, Y\in I$.
In addition, there holds 
\begin{equation}\label{eq:r12}
r_{\mathrm{i},1}(\varepsilon;t+2) =  r_{\mathrm{i},1}(\varepsilon;t), \quad \tilde r_{\mathrm{i},2}(\varepsilon;t+2) =  \tilde r_{\mathrm{i},2}(\varepsilon;t) \quad \mbox{for} \; 0\le t<2.
\end{equation}
\end{lemma}
The proof follows the same lines as Lemma 3.1 of \cite{lin_zhang17}, and we omit it. \\

Define the kernels
\begin{eqnarray}
&& \hspace*{35pt} \rho(X, Y) = \frac{1}{\pi} \left[ \ln |X-Y| + \ln \abs{\sin \left(\frac{\pi(X-Y)}{2}\right)}+ \ln\abs{\sin \left(\frac{\pi(X+Y+1)}{2}\right)}\right], \label{eq:rho} \\
&& \hspace*{35pt}  \rho_{\infty}(\kappa; X,Y) =  r_\mathrm{e}(\kappa,\varepsilon;X,Y) + r_{\mathrm{i},1}(\varepsilon; |X-Y|) + r_{\mathrm{i},2}(\varepsilon;|X+Y+1|),  \label{eq:rho_infty} \\
&& \hspace*{35pt} \tilde\rho_{\infty}(X,Y) =  \tilde r_{\mathrm{i},1}(\varepsilon; |X-Y|) + \tilde r_{\mathrm{i},2}(\varepsilon;|X+Y+1|),  \label{eq:rho_infty_tilde}
\end{eqnarray}
where $r_\mathrm{e}$,  $r_{\mathrm{i},1}$, $r_{\mathrm{i},2}$, $\tilde r_{\mathrm{i},1}$, and $\tilde r_{\mathrm{i},2}$ are given in \eqref{Ge_exp}, \eqref{Gi_exp}--\eqref{tGi_exp} respectively.
Let $S$, $S^{\infty}_\kappa$, $S^{\infty,\pm}_\kappa$ and $\tilde S^{\infty}$ be the integral operators defined over the interval $I$ and with the kernels $\rho(X,Y)$, $\rho_{\infty}(X,Y)$, $\rho^\pm(X,Y)$, and $\tilde \rho_\infty (X,Y)$:
\begin{eqnarray*}
&& [S \varphi](X) = \int_{I} \rho(X, Y)  \varphi(Y) \, dY,    \quad \quad \quad [S^{\infty}_\kappa \varphi](X) = \int_I  \rho_{\infty}(\kappa; X, Y)  \varphi(Y) \, dY; \\
&& [S^{\infty,\pm}_\kappa \varphi](X) = \int_I  \rho^\pm(X,Y)  \varphi(Y) \, dY,  \quad  [\tilde S^{\infty} \varphi](X) = \int_I   \tilde \rho_\infty (X,Y)  \varphi(Y) \, dY.
\end{eqnarray*}

\medskip

Define the functions spaces 
$$V_1 = \tilde H^{-\frac{1}{2}}(I):=\{ u=U|_{I}  \;\big{|}\;  U \in H^{-1/2}(\mathbf{R})\,\,  \mbox{and} \,\, supp \,U \subset \bar I   \} \quad \mbox{and} \quad V_2 =  H^{\frac{1}{2}}(I). $$
The following two lemmas hold for the operators defined above. The proof is provided in the appendix.
\begin{lemma} \label{lem:S-1}
If $\tilde{\varphi}(X)=\varphi(-X)$, then 
\begin{eqnarray*}
&& [S\tilde{\varphi}](X) = [S\varphi](-X),  \quad\quad\;\; [S^{\infty}_0\tilde{\varphi}](X) = [S^{\infty}_0\varphi](-X),  \\
&& [\tilde{S}^\infty\tilde{\varphi}](X) = [\tilde{S}^\infty\varphi](-X), \quad [S^{\infty,+}_0\tilde{\varphi}](X) = [S^{\infty,-}_0\varphi](-X).
\end{eqnarray*}
\end{lemma}

\vspace*{-10pt}

\begin{lemma} \label{lem:S-2}
The following holds for the operators $S$, $S^{\infty}_\kappa$, and $\tilde{S}^\infty$:
\begin{enumerate}
\item[(1)]
The operator $S$ is bounded from $V_1$ to $V_2$ with a bounded inverse.
Moreover,
\begin{equation}\label{alpha}
\alpha:=\langle S^{-1} 1, 1 \rangle_{L^2(I)} \neq 0.
\end{equation}
\item[(2)]
The operator $S+S^{\infty}_\kappa+\tilde S^{\infty}$ is invertible for small $\varepsilon$.
Let $\varphi$ and $\tilde\varphi$ be the solution of $$(S+S^{\infty}_0+\tilde S^{\infty})\varphi = g \quad \mbox{and} \quad (S+S^{\infty}_0+\tilde S^{\infty})\tilde{\varphi}=\tilde{g}$$ respectively,
where $\tilde{g}(X)=g(-X)$, then it holds that $\tilde{\varphi}(X)=\varphi(-X)$. The same holds for the operator $S+S^{\infty}_\kappa-\tilde S^{\infty}$.
\end{enumerate}
\end{lemma}

\medskip

We define the projection operator $P: V_1 \to V_2$ by
$$
P \varphi(X) = \langle \varphi, 1 \rangle 1,
$$
where $1$ is a function defined on the interval $I$ and is equal to one therein.  
We are ready to present the decomposition of the integral operators $T^\mathrm{e}+T^\mathrm{i} $, $T^{\mathrm{e},\pm}$, and $\tilde T^\mathrm{i} $ using the asymptotic expansion of the Green functions
in Lemmas \ref{lem:green_ext_asy} and \ref{lem:green_int_asy}.

\begin{proposition} \label{prop:decomp_op}
Let $k \in (0, 2\pi/d)$.
The operators $T^\mathrm{e}+T^\mathrm{i} $, $T^{\mathrm{e},\pm}$, and $\tilde T^\mathrm{i}$ admit the decompositions
$$T^\mathrm{e}+T^\mathrm{i}  = \beta P + S + S^\infty_\kappa, \quad T^{\mathrm{e},\pm} = \beta^\pm P + S^{\infty,\pm}_\kappa, \quad \tilde T^\mathrm{i} = \tilde\beta P + \tilde S^{\infty}, $$
where $S_\kappa^\infty$, $S^{\infty,\pm}_\kappa$ and $\tilde S^{\infty}$ are bounded from $V_1$ to $V_2$ with the operator norm
$$\| S_\kappa^{\infty}\|  \lesssim \varepsilon, \quad  \| S^{\infty,\pm}_\kappa\|  \lesssim \varepsilon,
\quad \mbox{and} \quad \| \tilde S^{\infty}\|  \lesssim e^{-1/\varepsilon} $$
uniformly in $\kappa$.
\end{proposition}

\section{Embedded eigenvalues and resonances}\label{sec:eig_res}

Define the singular frequencies of the scattering problem as the eigenvalues and resonances for the homogeneous problem.  Precisely, these are the $k$-values of pairs $(\kappa,k)$ for which the system (\ref{eq:Helmholtz}--\ref{eq:field2}) with the incident field $u^\mathrm{inc}$ removed has a nonzero solution, or, equivalently, pairs for which the system (\ref{eq:optT}) has a nonzero solution with~$\mathbf{f}=0$.  Eigenvalues are real values of $k$, whereas resonances are complex values of~$k$.
The field (eigenmode) corresponding to an eigenvalue decays exponentially above the grating, whereas the field (quasi-mode) corresponding to a resonance grows exponentially above the grating.

\subsection{The conditions for eigenvalues and resonances}\label{sec:res_cond}

In this section, we establish the condition for the singular frequencies.
From the previous discussions, we have seen that they are equal to the characteristic values $k$ of the system of integral operators~\eqref{eq:scattering2}.
When reduced to functions on the scaled interval~$I$, this amounts to finding those frequencies $k$ such that $\mathbb{T}\, \boldsymbol{\varphi}=0$ admits a nonzero solution
in~$(V_1)^4$.
Recall that 
\begin{equation*}
\mathbb{T}=\left[
\begin{array}{cc}
\hat{\mathbb{T}}  & \tilde{\mathbb{T}}  \\
\tilde{\mathbb{T}}  & \hat{\mathbb{T}}
\end{array}
\right], \quad\mbox{where} \;
\hat{\mathbb{T}}=\left[
\begin{array}{cc}
T^\mathrm{e}+T^\mathrm{i}    &  T^{\mathrm{e},-}   \\
 T^{\mathrm{e},+}     &  T^\mathrm{e}+T^\mathrm{i}  
\end{array}
\right]
\; \mbox{and} \;
\tilde{\mathbb{T}}=\left[
\begin{array}{cc}
 \tilde T^\mathrm{i} & 0  \\
 0 & \tilde T^\mathrm{i}  \\ 
\end{array}
\right].
\end{equation*}
We may decompose the set of its characteristic values as follows.

\begin{lemma}
Let $\mathbb{T}_+=\hat{\mathbb{T}}+\tilde{\mathbb{T}}$ and $\mathbb{T}_-=\hat{\mathbb{T}}-\tilde{\mathbb{T}}$.
Then
$$\sigma(\mathbb{T})=\sigma\big(\mathbb{T}_+\big)\, \cup \, \sigma\big(\mathbb{T}_-\big), $$
where $\sigma(\mathbb{T})$, $\sigma(\mathbb{T}_+)$ and $ \sigma(\mathbb{T}_-)$ denote the sets of
characteristic frequencies $k$ of $\mathbb{T}$, $\mathbb{T}_+$ and $\mathbb{T}_-$, respectively.
\end{lemma}

\begin{rem}\label{rmk:spectrum_op_T}
Such a decomposition of the spectrum follows from the symmetry of the grating geometry. In fact,
it can be shown that $\sigma\big(\mathbb{T}_+\big)$ (and $\sigma\big(\mathbb{T}_-\big)$ respectively) corresponds to the resonances
of the scattering problem where the lower half of the structure is replaced by a perfect conductor and 
the Neumann (Dirichlet) boundary condition is imposed over the lower slit aperture.
\end{rem}

\begin{proof} The function space $(V_1)^4$ can be decomposed as $(V_1)^4=V_\mathrm{even} \oplus V_\mathrm{odd}$,
where
$ V_\mathrm{even} = \{ \,[ \varphi_-, \varphi_+,  \varphi_-, \varphi_+ ]^T \,;\, \varphi_\pm \in V_1 \}  $
and
$ V_\mathrm{odd} = \{ \,[ \varphi_-, \varphi_+,  -\varphi_-, -\varphi_+ ]^T \,;\, \varphi_\pm \in V_1 \} $
are invariant spaces for~$\mathbb{T}$. Thus
$\sigma(\mathbb{T})= \sigma(\mathbb{T}|_{V_\mathrm{even}})\, \cup \, \sigma(\mathbb{T}|_{V_\mathrm{odd}})$.
Then observe that $\mathbb{T}[ \varphi_-, \varphi_+,  \varphi_-, \varphi_+ ]^T = [ \psi_-, \psi_+,  \psi_-, \psi_+ ]^T$,
with $\mathbb{T}_+[\varphi_-, \varphi_+]^T=[ \psi_-, \psi_+]^T$ so that $\sigma(\mathbb{T}|_{V_\mathrm{even}})=\sigma\big(\mathbb{T}_+\big)$, and similarly $\sigma(\mathbb{T}|_{V_\mathrm{odd}})=\sigma\big(\mathbb{T}_-\big)$.   
\end{proof}

\medskip

We now investigate the characteristic values of the operators $\mathbb{T}_+$ and $\mathbb{T}_-$.
They can be reduced to the roots of certain nonlinear functions.
We present the derivations for $\mathbb{T}_+$, and the derivations for $\mathbb{T}_-$
are parallel.

By defining the operators
\begin{equation*}
\mathbb{P}_\kappa= \left[
\begin{array}{cc}
(\beta+\tilde\beta) P    & \beta^- P \\
\beta^+ P  & (\beta+\tilde\beta) P 
\end{array}
\right],
\quad
\mathbb{S}^\infty_\kappa=
\left[
\begin{array}{cc}
S^\infty_\kappa + \tilde S^{\infty}    & S^{\infty,-}_\kappa \\
S^{\infty,+}_\kappa  & S^\infty_\kappa + \tilde S^{\infty}
\end{array}
\right],
\quad\mbox{and}\quad
\mathbb{L}_\kappa=S \mathbb{I}  +\mathbb{S}^\infty_\kappa,
\end{equation*}
and using the decomposition of the operators in Proposition~\ref{prop:decomp_op}, we obtain
$\mathbb{T}_+ = \mathbb{P}_\kappa + \mathbb{L}_\kappa$, and thus
$\mathbb{T}_+\, \boldsymbol{\varphi}=0$ becomes
\begin{equation}\label{eq:scattering4}
(\mathbb{P}_\kappa + \mathbb{L}_\kappa) \boldsymbol{\varphi} =   \mathbf{0}, 
\end{equation}
where $\boldsymbol{\varphi} = [\, \varphi^-,  \varphi^+ \,] ^T$.
Let $\{\mathbf{e}_j\}_{j=1}^2\in V_1 \times V_1 $  be given by $\mathbf{e}_1 = [1, 0]^T $ and $\mathbf{e}_2 = [0, 1]^T $.

\begin{lemma} \label{lem:L_inv} 
$\mathbb{L}_\kappa$ is invertible for sufficiently small  $\varepsilon$, and there holds
\begin{eqnarray*} 
&& \mathbb{L}_\kappa^{-1}  \mathbf{e}_1 = ( S^{-1}1 ) \cdot \mathbf{e}_1 + O(\varepsilon), \quad
\mathbb{L}_\kappa^{-1}  \mathbf{e}_2 = ( S^{-1}1 ) \cdot \mathbf{e}_2 + O(\varepsilon); \\
&& \langle \mathbb{L}_\kappa^{-1}  \mathbf{e}_1,  \mathbf{e}_1 \rangle = \alpha +  O(\varepsilon), \quad \quad
\langle \mathbb{L}_\kappa^{-1}  \mathbf{e}_1,  \mathbf{e}_2\rangle =O(\varepsilon)
\end{eqnarray*}
uniformly in $\kappa$, where $\alpha$ is defined in \eqref{alpha}.
\end{lemma}

\begin{proof} By Lemma~\ref{lem:S-2}, $\mathbb{L}_\kappa$ is invertible for small enough  $\varepsilon$ via the Neumann series,
$$ \mathbb{L}_\kappa^{-1}  = \left(S\mathbb{I}+\mathbb{S}^\infty_\kappa\right)^{-1} = \left(\sum_{j=0}^\infty (-1)^j\left(S^{-1}\mathbb{S^\infty_\kappa} \right)^j \right)  S^{-1} = S^{-1}\mathbb{I} + O(\varepsilon). $$
The assertion holds from the definition \eqref{alpha} of the constant~$\alpha$.
\end{proof}

\medskip

Applying $\mathbb{L}_\kappa^{-1}$ on both sides of \eqref{eq:scattering4} yields
\begin{equation}\label{eq:scattering5}
\mathbb{L}_\kappa^{-1} \;\mathbb{P}_\kappa \; \boldsymbol{\varphi} + \boldsymbol{\varphi} = \mathbf{0},
\end{equation}
which, using the equation
 $$  \mathbb{P}_\kappa \; \boldsymbol{\varphi} = (\beta+\tilde \beta)  \langle  \boldsymbol{\varphi}, \mathbf{e}_1 \rangle \mathbf{e}_1  + \beta^- \langle  \boldsymbol{\varphi}, \mathbf{e}_2 \rangle \mathbf{e}_1  
 + \beta^+ \langle  \boldsymbol{\varphi}, \mathbf{e}_1 \rangle \mathbf{e}_2  + (\beta+\tilde \beta) \langle \boldsymbol{\varphi}, \mathbf{e}_2 \rangle \mathbf{e}_2,$$
can be expanded into
\begin{equation*}
(\beta+\tilde \beta) \langle  \boldsymbol{\varphi}, \mathbf{e}_1 \rangle  \mathbb{L}_\kappa^{-1}  \mathbf{e}_1  + 
\beta^- \langle  \boldsymbol{\varphi}, \mathbf{e}_2 \rangle \mathbb{L}_\kappa^{-1} \mathbf{e}_1  + \beta^+ \langle  \boldsymbol{\varphi}, \mathbf{e}_1 \rangle \mathbb{L}_\kappa^{-1} \mathbf{e}_2
+ (\beta+\tilde \beta)  \langle  \boldsymbol{\varphi}, \mathbf{e}_2 \rangle  \mathbb{L}_\kappa^{-1}  \mathbf{e}_2 +
\boldsymbol{\varphi} =   \mathbf{0}. 
\end{equation*}
Taking the $L^2$-inner product of the above equation with $\mathbf{e}_1$ and $\mathbf{e}_2$ yields
\begin{equation*}
\tilde{\mathbb{M}}_{\kappa,+} \left[
 \begin{array}{llll}
\langle  \boldsymbol{\varphi}, \mathbf{e}_1 \rangle  \\
\langle  \boldsymbol{\varphi}, \mathbf{e}_2 \rangle
\end{array}
\right] =
\left[
\begin{array}{llll}
0   \\
0
\end{array}
\right],
\end{equation*}
 where the matrix $\tilde{\mathbb{M}}_{\kappa,+}$ is defined by
 \begin{equation}  \label{eq:matrix-M_tilde+}
 \tilde{\mathbb{M}}_{\kappa,+}:=
 \left[
 \begin{array}{cc}
  \langle \mathbb{L}_{\kappa}^{-1}  \mathbf{e}_1,  \mathbf{e}_1 \rangle  &  \langle \mathbb{L}_{\kappa}^{-1}  \mathbf{e}_2,  \mathbf{e}_1\rangle \\
 \langle \mathbb{L}_{\kappa}^{-1}  \mathbf{e}_1,  \mathbf{e}_2 \rangle  &  \langle \mathbb{L}_{\kappa}^{-1}  \mathbf{e}_2,  \mathbf{e}_2 \rangle 
 \end{array}
 \right] 
  \left[
 \begin{array}{cc}
 \beta+ \tilde\beta &  \beta^-\\
 \beta^+  &  \beta + \tilde\beta
  \end{array}
 \right]  
 +  \left[
 \begin{array}{llll}
  1  &  0 \\
  0 &  1
  \end{array}
 \right]. 
 \end{equation}

Let $\tilde\lambda_{1,+}(k; \kappa, \varepsilon)$ and $\tilde\lambda_{2,+}(k; \kappa, \varepsilon)$ be the eigenvalues of $\tilde{\mathbb{M}}_{\kappa,+}$.
From the above discussions, it is seen that the characteristic values of the operator-valued function $\mathbb{T}_+(k;\kappa,\varepsilon)$ are the roots of $\tilde\lambda_{1,+}(k)$ and $\tilde\lambda_{2,+}(k)$.
Following a similar decomposition for $\mathbb{T}_-$, then
the characteristic values of the operator-valued function $\mathbb{T}_-$ are the roots of $\tilde\lambda_{1,-}(k)$ and $\tilde\lambda_{2,-}(k)$, which are eigenvalues of the matrix
\begin{equation}  \label{eq:matrix-M_tilde-}
\tilde{\mathbb{M}}_{\kappa,-}:=
\left[
\begin{array}{cc}
 \langle \mathbb{L}_{\kappa}^{-1}  \mathbf{e}_1,  \mathbf{e}_1 \rangle  &  \langle \mathbb{L}_{\kappa}^{-1}  \mathbf{e}_2,  \mathbf{e}_1\rangle \\
 \langle \mathbb{L}_{\kappa}^{-1}  \mathbf{e}_1,  \mathbf{e}_2 \rangle  &  \langle \mathbb{L}_{\kappa}^{-1}  \mathbf{e}_2,  \mathbf{e}_2 \rangle 
 \end{array}
\right] 
 \left[
\begin{array}{cc}
 \beta - \tilde\beta &  \beta^- \\
 \beta^+ &  \beta - \tilde\beta
 \end{array}
\right]  
+  \left[
\begin{array}{llll}
 1  &  0 \\
 0 &  1
 \end{array}
\right].
\end{equation}
\begin{rem}
The operator $\mathbb{P}_\kappa$  and $\mathbb{L}_\kappa$ take different forms in the decomposition of $\mathbb{T}_+$ and~$\mathbb{T}_-$.  For $\mathbb{T}_-$, all quantities with a tilde in the definitions of $\mathbb{P}_\kappa$ and $\mathbb{S}_\kappa^\infty$ should be multiplied by $-1$.  
We suppress this dependence here and henceforth, as it is clear in context.
The dependence on $\kappa$ is retained because the study of embedded eigenvalues and associated resonance is an analysis of the behavior of the scattering problem for $\kappa$ at and near~$0$.\\
\end{rem}

Since the leading-order of $\beta$ in $\varepsilon$ is $O(1/\varepsilon)$, we scale the matrix $\tilde{\mathbb{M}}_{\kappa,\pm}$ by letting
\begin{equation}\label{eq:matrix-M}
\mathbb{M}_{\kappa,\pm} \,:=\, \varepsilon\, \tilde{\mathbb{M}}_{\kappa,\pm},
\end{equation}
and the eigenvalues of $\mathbb{M}_{\kappa,\pm}$ are
\begin{equation}\label{eq:lambdas}
\lambda_{j,\pm}(k; \kappa, \varepsilon) \,:=\, \varepsilon \, \tilde\lambda_{j,\pm}(k; \kappa, \varepsilon), \quad j=1,2.
\end{equation}
The following proposition summarizes the resonance condition.

\begin{proposition} \label{prop:res_cond}
The singular frequencies of the scattering problem \eqref{eq:Helmholtz}--\eqref{eq:field2}  are the roots of the functions $\lambda_{j,\pm}(k; \kappa, \varepsilon)$, $(j=1,2)$, 
where $\lambda_{1,\pm}$ and $\lambda_{2,\pm}$ are eigenvalues of the matrix $\mathbb{M}_{\kappa,\pm}$.
\end{proposition}

\subsection{Embedded eigenvalues and resonances for $\kappa=0$}\label{sec:kappa=0}

We investigate the roots of the functions $\lambda_{j,\pm}(k; \kappa, \varepsilon)$ $(j=1,2)$ when $\kappa=0$.
It is shown that real-valued roots and complex-valued roots with negative imaginary part coexist.  They correspond respectively to eigenvalues and resonances of the scattering problem at normal incidence. We prove their existence and derive their
asymptotic expansions.

\begin{lemma} \label{lem:identity-L}
The following statements for $\kappa=0$ hold.
\begin{itemize}
\item[(i)]  If $k$ is real, then $\mathbb{L}_0^{-1}\boldsymbol{\varphi}$ is real for any real-valued function $\boldsymbol{\varphi}$.
\item[(ii)]  $ \langle \mathbb{L}^{-1}_0  \mathbf{e}_1,  \mathbf{e}_1 \rangle = \langle \mathbb{L}^{-1}_0  \mathbf{e}_2,  \mathbf{e}_2 \rangle$ and
$ \langle \mathbb{L}^{-1}_0  \mathbf{e}_1,  \mathbf{e}_2 \rangle = \langle   \mathbb{L}^{-1}_0 \mathbf{e}_2,  \mathbf{e}_1 \rangle. $
\end{itemize}
\end{lemma}

\begin{proof} First,  in view of  Lemmas \ref{lem:green_ext_asy} and \ref{lem:green_int_asy},
 $\mathbb{L}_0^{-1}\boldsymbol{\varphi}$ is real-valued since the kernels of operators $S$, $S_0^\infty$, $\tilde S^{\infty}$ $S_{0}^{\infty,-}$,
and $S_{0}^{\infty,+}$ are all real, and assertion (i) follows.

To show (ii), let $\boldsymbol{\varphi}=(\varphi_1, \varphi_2)^T$ and $\tilde{\boldsymbol{\varphi}}=(\tilde{\varphi}_1, \tilde{\varphi}_2)^T$ satisfy
 $\mathbb{L}_0\boldsymbol{\varphi} =\mathbf{e}_1$ and $\mathbb{L}_0\tilde{\boldsymbol{\varphi}} =\mathbf{e}_2 $. 
By a direct calculation, it is seen that
 $$ (\hat S-S^{\infty,-}_0 \hat S^{-1} S^{\infty,+}_0)\varphi_1 = 1, \quad (\hat S-S^{\infty,+}_0 \hat S^{-1} S^{\infty,-}_0)\tilde{\varphi}_2 = 1,   $$
where $\hat S:= S+S_0^\infty+\tilde S^{\infty}$.
By virtue of Lemmas \ref{lem:S-1} and \ref{lem:S-2}, there holds $\tilde{\varphi}_2(X)=\varphi_1(-X)$, and it follows that 
$\langle \mathbb{L}^{-1}_0  \mathbf{e}_1,  \mathbf{e}_1 \rangle = \langle \mathbb{L}^{-1}_0  \mathbf{e}_2,  \mathbf{e}_2 \rangle$.
Similarly, it can be shown that  $\tilde{\varphi}_1(X)=\varphi_2(-X)$, so the second identity also holds. 
\end{proof} 

\medskip

When $\kappa=0$, by noting that $\beta_\mathrm{e}^\pm (k, 0) =\hat\beta(k)$ (see~\eqref{eq:beta_kappa0})
and using the equalities in Lemma~\ref{lem:identity-L}, the matrix $\tilde{ \mathbb{M}}_{0,\pm}$ can be expressed as
\begin{equation} 
\tilde{\mathbb{M}}_{0,\pm}:=
\left[
\begin{array}{cc}
 \langle \mathbb{L}_0^{-1}  \mathbf{e}_1,  \mathbf{e}_1 \rangle  &  \langle \mathbb{L}_0^{-1}  \mathbf{e}_1,  \mathbf{e}_2\rangle \\
 \langle \mathbb{L}_0^{-1}  \mathbf{e}_1,  \mathbf{e}_2 \rangle  &  \langle \mathbb{L}_0^{-1}  \mathbf{e}_2,  \mathbf{e}_2 \rangle 
 \end{array}
\right] 
 \left[
\begin{array}{cc}
 \beta \pm \tilde\beta &  \hat\beta \\
\hat\beta &  \beta \pm \tilde\beta
 \end{array}
\right]  
+  \left[
\begin{array}{cc}
 1  &  0 \\
 0 &  1
 \end{array}
\right]. 
\end{equation}
It can be calculated that the eigenvalues of $\tilde{\mathbb{M}}_{0,\pm}$ are
\begin{eqnarray}\label{eq:eigen_M_tilde_pm}
\tilde\lambda_{1,\pm}(k; 0,\varepsilon) &=& 1+(\beta \pm \tilde \beta+ \hat \beta)   \left(\langle \mathbb{L}^{-1}_0  \mathbf{e}_1,  \mathbf{e}_1 \rangle  + \langle \mathbb{L}^{-1}_0  \mathbf{e}_1,  \mathbf{e}_2\rangle\right), \label{eq:lambda1}\\
\tilde \lambda_{2,\pm}(k; 0,\varepsilon) &=& 1+(\beta \pm \tilde \beta- \hat \beta )   \left(\langle \mathbb{L}^{-1}_0  \mathbf{e}_1,  \mathbf{e}_1 \rangle  - \langle \mathbb{L}^{-1}_0  \mathbf{e}_1,  \mathbf{e}_2\rangle\right),
\label{eq:lambda2} 
\end{eqnarray}
and the associated eigenvectors are $[1 \;\; 1]^T$ and $[1 \;\; -1]^T$.
Therefore, in view of Lemma~\ref{lem:L_inv} and formulas \eqref{beta_e}--\eqref{gamma} for $\beta$ and $\tilde \beta$, the eigenvalues of $\mathbb{M}_{0,\pm}$ are expressed explicitly as 
\begin{eqnarray}
\hspace{30pt}\lambda_{1,\pm}(k; 0, \varepsilon)&=&\varepsilon +\left[ \frac{1}{k\tan k}  \pm \frac{1 }{k\sin k}  + \frac{1}{\pi} \varepsilon \ln \varepsilon + \varepsilon \gamma(k, 0)  +  \varepsilon \hat\beta(k)  \right]   \left(\alpha + O(\varepsilon) \right),  \label{eq:formula_lambda1}  \\
\hspace{30pt} \lambda_{2,\pm}(k; 0, \varepsilon)&=&\varepsilon +\left[ \frac{1}{k\tan k} \pm \frac{1 }{k\sin k}  + \frac{1}{\pi} \varepsilon \ln \varepsilon + \varepsilon \gamma(k, 0) -  \varepsilon \hat\beta(k)  \right]   \left(\alpha + O(\varepsilon) \right).
\label{eq:formula_lambda2}  
\end{eqnarray}

\begin{rem}\label{rmk:imag_lambda}
If $(0,k)\in D_1$ (see~(\ref{diamond})), that is $0<k<2\pi/d$, then from the explicit expressions \eqref{gamma} and \eqref{eq:beta_kappa0}, we see that
$$ \mbox{Im} \, \Big(\gamma(k, 0)  +  \hat\beta(k)\Big) = - \frac{2}{d}  \frac{1}{\zeta_0(k)}
\quad \mbox{and} \quad 
 \mbox{Im} \, \Big(\gamma(k, 0)  -  \hat\beta(k)\Big) = 0.
$$ 
The $O(\varepsilon)$ terms in \eqref{eq:formula_lambda1} and \eqref{eq:formula_lambda2} are real-valued by Lemma~\ref{lem:identity-L}, and hence 
$\lambda_{1,+}(k; 0, \varepsilon)$ and $\lambda_{1,-}(k; 0, \varepsilon)$  are complex-valued functions, while $\lambda_{2,+}(k; 0, \varepsilon)$ and $\lambda_{2,-}(k; 0, \varepsilon)$ are real-valued functions. 
\end{rem}

\medskip

\begin{theorem}\label{thm:asym_res_eig}
If $\kappa=0$, the singular frequencies of  the scattering problem \eqref{eq:Helmholtz}--\eqref{eq:field2} admit the following asymptotic expressions in $\varepsilon$:
\begin{eqnarray*}
k_m^{(1)} &=&  m\pi +  2m\pi  \left [ \frac{1}{\pi}\varepsilon \ln\varepsilon + \left(  \frac{1}{\alpha} +  \gamma(m \pi, 0) +   \hat\beta(m \pi)  \right) \varepsilon \right] + O(\varepsilon^2\ln^2\varepsilon); \\
k_m^{(2)} &=&  m\pi +  2m\pi  \left [ \frac{1}{\pi}\varepsilon \ln\varepsilon + \left(  \frac{1}{\alpha} +  \gamma(m \pi, 0) -   \hat\beta(m \pi)  \right) \varepsilon \right] + O(\varepsilon^2\ln^2\varepsilon)
\end{eqnarray*} 
for positive integers $m<2/d$.  In the above, $\Im \, k_m^{(1)}=O(\varepsilon)$ and $\Im \, k_m^{(2)}=0$.
\end{theorem}

\begin{rem}
The frequencies $ k_m^{(1)}$ for $m<2/d$ are resonances in the lower half of the complex plane, which are also called Fabry-Perot resonances \cite{garcia10, lin_zhang17, lin_zhang18_1}.
The frequencies $k_m^{(2)}$ are real-valued eigenvalues embedded in the continuous spectrum,
since the continuous spectrum of the quasi-periodic scattering operator is $[0,\infty)$ when $\kappa=0$ \cite{bonnet_starling94, shipman10}.
\end{rem}

\begin{rem}\label{rmk:dis_m}
The asymptotic expansions of  $k_m^{(1)}$ and $k_m^{(2)}$ in Theorem~\ref{thm:asym_res_eig} still hold for $m>2/d$.
However, when $m>2/d$, one has $\Re\,k_m^{(2)}\geq2\pi/d$  so that $(0,\Re\,k_m^{(2)})\not\in D_1$.  That is, this wavenumber-frequency pair lies above the diamond region in which exactly one of the Rayleigh modes is propagating and embedded eigenvalues are not expected.
Indeed, for such singular frequencies, there holds $\mbox{Im} \, \Big(\gamma(m\pi, 0)  -  \hat\beta(m\pi)\Big) \neq 0$ and the $O(\varepsilon^2\ln^2\varepsilon)$-term is also complex, and
the frequencies $k_m^{(2)}$ become complex-valued resonances with $\Im \, k_m^{(2)}=O(\varepsilon)$.
Here we restrict our attention to $m<2/d$ since we are concerned with embedded eigenvalues in this paper.
\end{rem}

\medskip

\begin{proof}  The leading-order term $\frac{1}{k\tan k}  + \frac{1 }{k\sin k}$ of $\lambda_{1,+}(k)$ attains a simple root $k_{m,0}=m\pi$ for odd integers $m$.
Let us choose the disc $B_{\delta}(k_{m,0})$ with radius $\delta$ centered at $k_{m,0}$ in the complex $k$-plane, with $\delta=O(1)$ as $\varepsilon\to0$. 
We analytically extend the functions $\beta_\mathrm{e}(k)$,  $\beta_\mathrm{i}(k)$,  $\tilde \beta $ and $\hat \beta(k)$ to $B_{\delta}(k_{m,0})$.  
One can show that the asymptotic expansions in $\varepsilon$ for the kernels $G_\varepsilon^\mathrm{e}$, $G_\varepsilon^{\mathrm{e},\pm}$,  $G_\varepsilon^\mathrm{i}$ and $\tilde G_\varepsilon^\mathrm{i}$ given in Lemmas \ref{lem:green_ext_asy}
 and \ref{lem:green_int_asy} hold in $B_{\delta_0}(k_{m,0})$.
From Rouche's theorem, we deduce that there is a simple root  of $\lambda_{1,+}(k; 0, \varepsilon)$, denoted as $k_{m}^{(1)}$,  close to $k_{m,0}$
if $\varepsilon$ is sufficiently small. 

To obtain the leading-order asymptotic terms of $k_{m}^{(1)}$, let us consider the root of
\begin{equation*}
\hat \lambda_{1,+}(k; \varepsilon)=\varepsilon +\left[ \frac{1}{k\tan k}  + \frac{1 }{k\sin k} + \frac{1}{\pi} \varepsilon \ln \varepsilon + \varepsilon \gamma(k, 0)  + \varepsilon \hat\beta(k)  \right] \alpha.
\end{equation*}
Let $\delta k = k-k_{m,0}$, then the Taylor expansion for $\hat \lambda_{1,+}(k, \varepsilon)$ at $k=k_{m,0}$ yields
\begin{eqnarray*}
\hat \lambda_{1,+}(k; \varepsilon) &=& \varepsilon +\bigg[\left.\left(\frac{1}{k\tan k}  + \frac{1 }{k\sin k}\right)'\right|_{k=k_{m,0}}\cdot\delta k +O(\delta k^2) +
\frac{1}{\pi} \varepsilon \ln \varepsilon \\
& &+ \varepsilon (\gamma(k_{m,0}, 0)+\hat{\beta}(k_{m,0}))  + \varepsilon \cdot O(\delta k)  \bigg] \cdot \alpha\,,
\end{eqnarray*}
and the root of $\hat \lambda_{1,+}$ is given by
$$
\hat k_{m}^{(1)} =k_{m, 0}+ 2m\pi \left [ \frac{1}{\pi}\varepsilon \ln\varepsilon + \left(  \frac{1}{\alpha} + \gamma(k_{m,0},0) + \hat\beta(k_{m,0}) \right)\varepsilon \right] + O(\varepsilon^2\ln^2\varepsilon). 
$$

The high-order term of the roots for $\lambda_{1,+}(k)$ can be obtained by the Rouche's theorem.
Note that 
$\lambda_{1,+}(k) - \hat\lambda_{1,+}(k)  =  (\hat\lambda_{1,+}(k) - \varepsilon)  \cdot O(\varepsilon)$,
one can find a constant $C>0$ such that 
$| \lambda_{1,+}(k) - \hat\lambda_{1,+}(k)|  < | \hat\lambda_{1,+}(k)|$ 
for all $k$ satisfying $|k-\hat k_{m}^{(1)}| = C \varepsilon^2\ln^2\varepsilon$.  The assertion holds by Rouche's theorem.

The roots of $\lambda_{1,-}(k; 0, \varepsilon)$ and $\lambda_{2,\pm}(k; 0, \varepsilon)$ can also be obtained by perturbation arguments.
In particular, $\lambda_{2,+}(k)$ attains roots close to $m\pi$ with odd integers $m$, while $\lambda_{1,-}(k)$ and $\lambda_{2,-}(k)$ attain roots close to $m\pi$ with even integers $m$.
Finally, $k_m^{(2)}$ are seen to be real by noting that the functions $\lambda_{2,+}(k; 0, \varepsilon)$ and  $\lambda_{2,-}(k; 0, \varepsilon)$ are real-valued for $k\in(0,2\pi/d)$.  
\end{proof}

\subsection{Perturbation of embedded eigenvalues and resonances}\label{sec:resonances}

If $\kappa \neq 0 $, both the eigenvalues and resonances at $\kappa=0$ will be perturbed on the complex $k$-plane.
The resonances will stay in the lower half plane. On the other hand, the real eigenvalues will emerge as a second group of complex-valued resonances that enter the lower half plane. We aim to obtain the asymptotic expansion of these two groups
of resonances. In particular, we would like to characterize the order of the imaginary parts for the resonances that originate
from the perturbation of embedded eigenvalues.

\medskip
\begin{rem}
We shall assume that $\kappa = O(\varepsilon^\rho)$ in this section, where $0<\rho<\frac{1}{2}$. 
This is not an essential assumption for the expansion of resonances discussed
in what follows.  However, one would need to investigate higher-order terms more thoroughly in the asymptotic expansion if $\rho\ge \frac{1}{2}$.
\end{rem}
\medskip

A brute-force perturbation argument leads to an order of $O(\kappa\varepsilon)$ for the imaginary parts 
of resonances that emanate from the eigenvalues as $\kappa$ is perturbed from $0$.
To obtain a better expansion, as given in Theorem~\ref{thm:asym_eig_perturb} below, we need to exploit the symmetry of the matrices~$\mathbb{M}_{\kappa,\pm}$. To this end,
we define matrices $\hat{\mathbb{M}}_{\kappa,+}$ and  $\hat{\mathbb{M}}_{\kappa,-}$ by
\begin{equation}  \label{eq:matrix-M_hat}
\frac{1}{\varepsilon}\, \hat{\mathbb{M}}_{\kappa,\pm}
\;=\;
\alpha
 \left[
\begin{array}{cc}
 \beta \pm \tilde \beta &  \beta^- \\
\beta^+ &  \beta \pm \tilde \beta
 \end{array}
\right]  
+  \left[
\begin{array}{llll}
 1  &  0 \\
 0 &  1
 \end{array}
\right],
\end{equation}
where $\alpha:=\langle S^{-1} 1, 1 \rangle $. The eigenvalues of the matrix $\hat{\mathbb{M}}_{\kappa,\pm}$ are
\begin{eqnarray}\label{eq:eigenval_M_hat}
\hat\lambda_{1,\pm}(k; \kappa,\varepsilon) &=& \varepsilon+\varepsilon\alpha (\beta \pm \tilde \beta + \sqrt{\beta^- \cdot \beta^+} ),  \\
\hat\lambda_{2,\pm}(k; \kappa,\varepsilon) &=& \varepsilon+\varepsilon\alpha (\beta \pm \tilde \beta - \sqrt{\beta^- \cdot \beta^+} ).
\end{eqnarray}
The corresponding (right) eigenvectors of $\hat{\mathbb{M}}_{\kappa,\pm}$ are
\begin{equation}\label{eq:eigenvec_M_hat}
\hat v_{1,\pm} = [\; 1, \;\; \sqrt{\beta^- \cdot \beta^+} / \beta^- \; ] ^T, \quad
\hat v_{2,\pm} = [\; 1, \;\; -\!\sqrt{\beta^- \cdot \beta^+} / \beta^- \; ] ^T,
\end{equation}
and the left eigenvectors are
\begin{equation}\label{eq:left_eigenvec_M_hat}
\hat w_{1,\pm} = [\; 1/2, \;\;  \beta^-/2(\sqrt{\beta^- \cdot \beta^+}) \; ], \quad
\hat w_{2,\pm} = [\; 1/2, \;\; -\!\beta^-/(2\sqrt{\beta^- \cdot \beta^+}) \;].
\end{equation}

\begin{lemma}\label{lem:eigen_M_hat}
The eigenvalues and eigenvectors of $\hat{\mathbb{M}}_{\kappa,\pm}$ attain the following asymptotic expansions as $\kappa\to0$.
\begin{eqnarray}
&& \hspace{10pt}\hat\lambda_{j,\pm}(k; \kappa,\varepsilon) = \varepsilon+\varepsilon \alpha \left[\beta \pm \tilde\beta + \textstyle\frac{1}{2}(-1)^{j+1} (\beta^+ + \beta^-) \right]  + O(\kappa^2\varepsilon), \quad j=1,2;  \label{eq:lambdaj} \\
&& \hspace{10pt} \hat v_{j,\pm} = [\; 1, \;\;  (-1)^{j+1} (1+\eta) \; ] ^T+ O(\kappa^2), \quad j=1,2; \label{eq:exp_vj} \\
&& \hspace{10pt} \hat w_{j,+} = [\; 1/2, \;\;  (-1)^{j+1}/(2(1+\eta)) \; ] + O(\kappa^2), \quad j=1,2,  \label{eq:exp_wj}
\end{eqnarray}
where $ \eta =  O(\kappa)$.
\end{lemma}

\begin{proof}  Note that 
\begin{equation*}
\beta^- \cdot \beta^+ = \textstyle\frac{1}{4} \left[ (\beta^+ + \beta^-)^2 - (\beta^+ - \beta^-)^2 \right ].
\end{equation*}
Since $\beta^+(k,0) = \beta^-(k,0) = \hat\beta$, we have $(\beta^+ - \beta^-)^2 = O(\kappa^2)$ as $\kappa\to0$.
Hence
$$ \sqrt{\beta^- \cdot \beta^+} = \textstyle\frac{1}{2}  (\beta^+ + \beta^-) + O(\kappa^2),   $$
and the expansions of the eigenvalues and eigenvectors follow.
\end{proof}

If $(\kappa,k)\in D_1$,
it follows from the explicit expressions \eqref{beta_e}--\eqref{gamma} that 
\begin{eqnarray}
&& \Im \, \beta(k, \kappa) + \frac{1}{2}(\Im \, \beta^+(k,\kappa) + \Im \, \beta^-(k,\kappa))  = O(1), \label{eq:imag_lambda1}  \\
&& \Im \, \beta(k, \kappa) - \frac{1}{2}(\Im \, \beta^+(k,\kappa) + \Im \, \beta^-(k,\kappa))  = \frac{\cos(\kappa d_0-1)}{\zeta_0(k) \cdot d}  = O(\kappa^2).  \label{eq:imag_lambda2}
\end{eqnarray}
Consequently, we have the following proposition deduced from the previous lemma.

\begin{proposition}\label{prop:perturb_eig2_M}
If $(\kappa,k)\in D_1$, then $\Im \, \hat\lambda_{1,\pm}(k) = O(\varepsilon) $ and $\Im \, \hat\lambda_{2,\pm}(k) = O(\kappa^2\varepsilon)$.
\end{proposition}

\medskip

Next, we prove the key lemma for the sensitivity of eigenvalues of the matrix $\mathbb{M}_{\kappa,\pm}$ with respect to 
the perturbation $\delta  \mathbb{M}_\pm := \mathbb{M}_{\kappa,\pm} - \hat{\mathbb{M}}_{\kappa,\pm}$ .
\begin{lemma}\label{lem:perturb_eig_M}
Let $\{\lambda_{j,\pm}\}_{j=1}^2$ and $\{\hat\lambda_{j,\pm}\}_{j=1}^2$ be the eigenvalues of $\mathbb{M}_{\kappa,\pm}$ and $\hat{\mathbb{M}}_{\kappa,\pm}$ respectively, then
\begin{equation}\label{eq:lambda_sensitivity}
\lambda_{j,\pm} (k; \kappa,\varepsilon)  = (1+r_j(k; \kappa,\varepsilon) ) \cdot \hat\lambda_{j,\pm} (k; \kappa,\varepsilon)  + r_{j,h} (k; \kappa,\varepsilon), \quad j=1,2,
\end{equation}
where $ r_j = O(\varepsilon)$ and  $ r_{j,h} = O(\varepsilon^2)$.
\end{lemma}

\begin{proof}  A direct comparison of \eqref{eq:matrix-M} and \eqref{eq:matrix-M_hat} gives
\begin{equation}\label{eq:delta_M}
\delta  \mathbb{M}_+= \varepsilon
\left[
\begin{array}{cc}
 \langle \mathbb{L}_\kappa^{-1}  \mathbf{e}_1,  \mathbf{e}_1 \rangle - \alpha &  \langle \mathbb{L}_\kappa^{-1}  \mathbf{e}_2,  \mathbf{e}_1\rangle \\
 \langle \mathbb{L}_\kappa^{-1}  \mathbf{e}_1,  \mathbf{e}_2 \rangle  &  \langle \mathbb{L}_\kappa^{-1}  \mathbf{e}_2,  \mathbf{e}_2 \rangle  - \alpha
 \end{array}
\right] 
 \left[
\begin{array}{cc}
 \beta+\tilde\beta &  \beta^- \\
\beta^+ &  \beta+\tilde\beta
 \end{array}
\right].
\end{equation}
Let  $\delta\lambda_{j,+}=\lambda_{j,+}(k; \kappa,\varepsilon) - \hat\lambda_{j,+}(k; \kappa,\varepsilon)$  and $\delta v_{j,+} =v_{j,+} - \hat v_{j,+}$ be the perturbation of the eigenvalues and eigenvectors.
It follows from Lemma~\ref{lem:L_inv} that 
\begin{equation}\label{eq:L2}
\langle \mathbb{L}_\kappa^{-1}  \mathbf{e}_\ell,  \mathbf{e}_j \rangle - \alpha \, \delta_{\ell j} = O(\varepsilon),
\end{equation}
where $\delta_{\ell j}$ is Kronecker delta, and consequently,
$\norm {\delta  \mathbb{M}_+} = O(\varepsilon)$.  An application of the Bauer-Fike theorem for the perturbation of eigenvalues (cf.~\cite{eistenstat}) yields
$$ |\delta\lambda_{j,+}| =   O(\varepsilon), \quad\mbox{and} \quad  \norm{\delta v_{j,+}} = O(\varepsilon). $$
Now from the relation $\mathbb{M}_{\kappa,+}  v_{j,+} = \lambda_{j,+} v_{j,+}$, we obtain
$$\hat \lambda_{j,+} \cdot \delta v_{j,+} + \delta\lambda_{j,+} \cdot \hat v_{j,+}  \;=\; \hat{\mathbb{M}}_{\kappa,+} \, \delta v_{j,+} +  \delta  \mathbb{M}_+ \, \hat v_{j,+} + O(\varepsilon^2).  $$
Multiplying by the left-eigenvector $\hat w_{j,+}$ leads to
\begin{equation}\label{eq:lambda_perturb1}
 \delta\lambda_{j,+}  \cdot (\hat w_{j,+} v_{j,+})  \;=\; \hat w_{j,+} \, \delta \mathbb{M}_\kappa  \, \hat v_{j,+}  + O(\varepsilon^2).
\end{equation}
Since $\hat \lambda_{j,+}$ is an eigenvalue of $\hat{\mathbb{M}}_{\kappa,+}$,  in light of \eqref{eq:delta_M}, we see that
\begin{equation}\label{eq:lambda_perturb2}
\delta \mathbb{M}_+  \, \hat v_{j,+}  \;=\; \frac{1}{\alpha}(\hat\lambda_{j,+} -\varepsilon)
\left[
\begin{array}{cc}
 \langle \mathbb{L}_\kappa^{-1}  \mathbf{e}_1,  \mathbf{e}_1 \rangle - \alpha &  \langle \mathbb{L}_\kappa^{-1}  \mathbf{e}_2,  \mathbf{e}_1\rangle \\
 \langle \mathbb{L}_\kappa^{-1}  \mathbf{e}_1,  \mathbf{e}_2 \rangle  &  \langle \mathbb{L}_\kappa^{-1}  \mathbf{e}_2,  \mathbf{e}_2 \rangle  - \alpha
 \end{array}
\right]  \hat v_{j,+}.
\end{equation}
The assertion follows from \eqref{eq:L2}--\eqref{eq:lambda_perturb2} and the expansion for the eigenvectors
in Lemma~\ref{lem:eigen_M_hat}.  The sensitivity of the eigenvalues for $\lambda_{j,-}$ is analyzed parallelly. 
\end{proof}

\begin{rem} If $k$ is real and $0<k<2\pi/d$, then in the above lemma, $r_2=O(\varepsilon) + i\,O(\kappa \varepsilon)$ and 
$r_{2,h}=O(\varepsilon^2)+i\,O(\kappa \varepsilon^2)$, where the $O (\cdot)$ terms are real. This can be shown by observing that
$ \langle \mathbb{L}_\kappa^{-1}  \mathbf{e}_\ell,  \mathbf{e}_j \rangle - \alpha \delta_{\ell j} = O(\varepsilon) + i \,O(\kappa \varepsilon)$
when $j=2$.
\end{rem}

Now with the explict expressions for the eigenvalues of $\hat{\mathbb{M}}_{\kappa,\pm}$ in Lemma \ref{lem:eigen_M_hat} and the sensitivity of eigenvalues for $\mathbb{M}_{\kappa,\pm}$ in Lemma \ref{lem:perturb_eig_M},
we are ready to present the perturbation of embedded eigenvalues and resonances when $\kappa$ becomes nonzero. This is given in the following theorem.

\medskip

\begin{theorem}\label{thm:asym_eig_perturb}
If $\kappa = O(\varepsilon^\rho)$ with $0<\rho<\frac{1}{2}$,
then the scattering problem \eqref{eq:Helmholtz}--\eqref{eq:field2} admits two groups of complex-valued resonances given by
\begin{eqnarray*}
k_m^{(1)} &=&  m \pi +  2m \pi  \left [ \frac{1}{\pi}\varepsilon \ln\varepsilon + \left(  \frac{1}{\alpha} +  \gamma(m \pi, \kappa) +   
\frac{1}{2}( \beta^+(m \pi,\kappa)) + \beta^-(m \pi, \kappa) \right)  \varepsilon \right] + O(\varepsilon^2\ln^2\varepsilon), \\
k_m^{(2)} &=&  m \pi + 2 m \pi  \left [ \frac{1}{\pi}\varepsilon \ln\varepsilon + \left(  \frac{1}{\alpha} +  \gamma(m \pi, \kappa) -   
\frac{1}{2}(\beta^+(m \pi,\kappa) + \beta^-(m \pi, \kappa) )  \right)\varepsilon \right] + O(\varepsilon^2\ln^2\varepsilon)
\end{eqnarray*} 
for $m<2/d$. Furthermore, there holds
$$ \Im \, k_m^{(1)} = O(\varepsilon) \quad \mbox{and} \quad \Im \, k_m^{(2)} = O(\kappa^2 \varepsilon). $$
\end{theorem}

\begin{proof}  With Lemmas \ref{lem:eigen_M_hat} and \ref{lem:perturb_eig_M},  the proof is analogous to that of Theorem~\ref{thm:asym_res_eig}.
First, from Rouche's theorem, there exists a simple root $k_{m}^{(j)}$ of $\lambda_{j,+}(k; \kappa, \varepsilon)$ for odd integer $m$
close to $k_{m,0}:=m\pi$, the root of the leading-order term $\frac{1}{k\tan k}+\frac{1}{k\sin k}$.  

To obtain the asymptotics of $k_{m}^{(j)}$, we first consider a root of $\hat \lambda_{j,+}(k; \kappa, \varepsilon)$, which is an eigenvalue of $\hat{\mathbb{M}}_{\kappa,+}$.
Note  that $\hat \lambda_{j,+}(k; \kappa, \varepsilon)$ attains the expansion \eqref{eq:lambdaj}. 
An application of the Taylor expansion for $\hat \lambda_{j,+}(k, \varepsilon)$ at $k=k_{m,0}$ yields
\begin{eqnarray*}
\hat \lambda_{j,+}(k; \varepsilon) &=& \varepsilon +\bigg[-1/(2k_{m,0}) \cdot \delta k +O(\delta k^2) +
\frac{1}{\pi} \varepsilon \ln \varepsilon + \varepsilon \gamma(k_{m,0}, \kappa)\\
&& + \varepsilon \cdot (-1)^{j+1} 
\cdot \frac{1}{2}\left(\beta^+(k_{m,0},\kappa) + \beta^-(k_{m,0},\kappa)\right) 
 +\varepsilon \cdot O(\delta k)  \bigg] \cdot \alpha + O(\kappa^2\varepsilon),
\end{eqnarray*}
where $\delta k:=k-k_{m,0}$.
In the above, it follows from \eqref{eq:imag_lambda1} and \eqref{eq:imag_lambda2} that
\begin{eqnarray}
&& \Im \, \gamma(k_{m,0}, \kappa) + \frac{1}{2}\left(\Im \, \beta^+(k_{m,0},\kappa) + \Im \, \beta^-(k_{m,0},\kappa)\right)  = O(1) ,   \label{eq:imag_gmma_beta1} \\
&& \Im \, \gamma(k_{m,0}, \kappa) - \frac{1}{2}\left(\Im \, \beta^+(k_{m,0},\kappa) + \Im \, \beta^-(k_{m,0},\kappa)\right)  = O(\kappa^2). \label{eq:imag_gmma_beta2}
\end{eqnarray}
Hence the root of $\hat \lambda_{j,+}$ can be expanded as
\begin{eqnarray*}
\hat k_{m}^{(j)} &=& k_{m, 0}+  2k_{m, 0} \left [ \frac{1}{\pi}\varepsilon \ln\varepsilon + \left(  \frac{1}{\alpha} + \gamma(k_{m,0},\kappa) + (-1)^{j+1} 
\cdot \frac{1}{2}\left(\beta^+(k_{m,0},\kappa) + \beta^-(k_{m,0},\kappa)\right)  \right)\varepsilon \right]  +  k_{m, h}^{(j)},
\end{eqnarray*}
in which  
\begin{eqnarray*}
&& \Re \, k_{m, h}^{(1)} = O(\varepsilon^2\ln^2\varepsilon), \quad \Im \, k_{m, h}^{(1)}= O(\varepsilon^2|\ln\varepsilon|), \\
&& \Re \, k_{m, h}^{(2)} = O(\varepsilon^2\ln^2\varepsilon),  \quad \Im \, k_{m, h}^{(2)} = O(\kappa^2\varepsilon).
\end{eqnarray*}

To obtain the high-order term of the roots, from Lemma~\ref{lem:perturb_eig_M} we have
$$  \lambda_{j,+}(k) - \hat\lambda_{j,+}(k)  \;=\;  O(\varepsilon)  \cdot   \hat\lambda_{j,+}(k)    + O(\varepsilon^2).   $$ 
Therefore, for a certain constant $C$, $| \lambda(k) - \hat\lambda_{j,+}(k)|  < | \hat\lambda_{j,+}(k)|$ holds
for all $k$ satisfying $|k-\hat k_{m}^{(j)}| = C \varepsilon^2\ln^2\varepsilon$,
and the asymptotic expansion of $k_{m}^{(j)}$ follows from Rouche's theorem.
Since $\kappa = O(\varepsilon^\rho)$ with $0<\rho<\frac{1}{2}$, in view of \eqref{eq:imag_gmma_beta1} and \eqref{eq:imag_gmma_beta2}, it follows that 
$\Im \, k_m^{(1)} = O(\varepsilon)$ and $\Im \, k_m^{(2)} = O(\kappa^2 \varepsilon)$. 
Finally, investigating the roots of $\lambda_{j,-}(k; \kappa, \varepsilon)$ yields resonances
close to $m\pi$ with even integers~$m$.
 \end{proof}

\section{Fano resonance and field enhancement}\label{sec:fano_field_enhancement}
We present quantitative analysis for the solution to the scattering problem \eqref{eq:Helmholtz} - \eqref{eq:field2}
and transmission anomaly through the slab near $\kappa=0$ and $k=\Re\,k_m^{(2)}$.
In particular, the expressions of reflected and transmitted field are obtained in Section~\ref{subsec:anomaly}, 
which allow for a rigorous proof of the presence of Fano resonance 
near the frequency $k=\Re\,k_m^{(2)}$. Additionally, we analyze quantitatively the field amplification at Fano resonance in Section~\ref{subsec:amplification}.

\subsection{Asymptotics of the solution to scattering problem}\label{subsec:scattering}

Decompose the system $\mathbb{T}\boldsymbol{\varphi}=\varepsilon^{-1}\mathbf{f}$ into its even and odd subsystems
$$  \mathbb{T}\boldsymbol{\varphi}_\mathrm{even}=\varepsilon^{-1}{\mathbf{f}}_\mathrm{even},  \quad   \mathbb{T}\boldsymbol{\varphi}_\mathrm{odd}=\varepsilon^{-1}{\mathbf{f}}_\mathrm{odd},   $$
in which
$\boldsymbol{\varphi}=\boldsymbol{\varphi}_\mathrm{even}+\boldsymbol{\varphi}_\mathrm{odd}$
and $\mathbf{f}=\mathbf{f}_\mathrm{even}+\mathbf{f}_\mathrm{odd}$, and
\begin{eqnarray*}
&& {\mathbf{f}}_\mathrm{even} = [f^-, \, f^+, \, f^-, \, f^+ ]^T, \quad {\mathbf{f}}_\mathrm{odd} = [f^-, \, f^+, \, -f^-, \, -f^+ ]^T,   \\
&& \boldsymbol{\varphi}_\mathrm{even} = [ \boldsymbol{\varphi}_{+}, \boldsymbol{\varphi}_{+}]^T,  \quad \quad  \quad \quad \;
      \boldsymbol{\varphi}_\mathrm{odd} = [ \boldsymbol{\varphi}_{-}, -\boldsymbol{\varphi}_{-}]^T.
\end{eqnarray*}
These two subsystems are equivalent to the two smaller systems
$$  \mathbb{T}_+\boldsymbol{\varphi}_{+} =  \varepsilon^{-1}\tilde{\mathbf{f}} \quad \mbox{and} \quad  \mathbb{T}_-\boldsymbol{\varphi}_{-} =  \varepsilon^{-1}\tilde{\mathbf{f}},   $$
where $\mathbb{T}_+=\hat{\mathbb{T}}+\tilde{\mathbb{T}}$ and $\mathbb{T}_-=\hat{\mathbb{T}}-\tilde{\mathbb{T}}$,  and $\tilde{\mathbf{f}} = [f^-, \, f^+]^T$.
Define
\begin{equation}\label{eq:mu_pm}
 \mu_+(\kappa) = e^{i\kappa d_0/2} + (1+\eta) e^{-i\kappa d_0/2} , \quad \mu_-(\kappa)=e^{i\kappa d_0/2} - (1+\eta) e^{-i\kappa d_0/2}, 
\end{equation}
where $\eta=O(\kappa)$ is defined in Lemma~\ref{lem:eigen_M_hat}.

\begin{lemma} \label{lem:phi}
The following asymptotic expansion holds for the solutions $\boldsymbol{\varphi}_+$ and  $\boldsymbol{\varphi}_-$ in $V_1\times V_1$:
\begin{eqnarray} 
\left[
\begin{array}{llll}
\langle  \boldsymbol{\varphi}_\pm, \mathbf{e}_1 \rangle  \\
\langle  \boldsymbol{\varphi}_\pm, \mathbf{e}_2 \rangle
\end{array}
\right]  \nonumber
&=& -\left(\alpha+ O(\varepsilon) \right)
\frac{\mu_+ }{2\lambda_{1,\pm}} 
\left(\left[
\begin{array}{cc}
\frac{1}{1+\eta} \\
1 
\end{array}
\right] + O(\varepsilon+\kappa^2) \right) \\
&& +\left(\alpha+ O(\varepsilon) \right)
\frac{\mu_- }{2\lambda_{2,\pm}} 
\left(\left[
\begin{array}{cc}
\frac{1}{1+\eta}  \\
-1 
\end{array}
\right] + O(\varepsilon+\kappa^2) \right), \label{eq:phi_1}
\quad
(|\kappa|,\varepsilon\to0)
\end{eqnarray} 
in which $\alpha:=\langle S^{-1} 1, 1 \rangle$.
In addition,
\begin{equation}\label{eq:phi_2}
\boldsymbol{\varphi}_\pm \;=\; \varepsilon^{-1} \mathbb{L}_{\kappa}^{-1}  \tilde{\mathbf{f}}
- \bigg[  \mathbb{L}_\kappa^{-1}  \mathbf{e}_1 \quad  \mathbb{L}_\kappa^{-1}  \mathbf{e}_2 \bigg]  
\left[
\begin{array}{cc}
 \beta + \tilde\beta&  \beta^- \\
\beta^+ &  \beta + \tilde\beta
 \end{array}
\right]\left[
\begin{array}{llll}
\langle  \boldsymbol{\varphi}_\pm, \mathbf{e}_1 \rangle  \\
\langle  \boldsymbol{\varphi}_\pm, \mathbf{e}_2 \rangle
\end{array}
\right].
\end{equation}
\end{lemma}

\begin{proof} We consider $\mathbb{T}_+\boldsymbol{\varphi}_{+} =  \varepsilon^{-1}\tilde{\mathbf{f}}$, and the proof for  
$\mathbb{T}_-\boldsymbol{\varphi}_{-} =  \varepsilon^{-1}\tilde{\mathbf{f}}$ is parallel.
Using the decomposition $\mathbb{T}_+=\mathbb{P}_\kappa+\mathbb{L}_\kappa$, the equation reads $(\mathbb{P}_\kappa+\mathbb{L}_\kappa)\boldsymbol{\varphi}_{+} = \varepsilon^{-1}\tilde{\mathbf{f}}$, which can be expressed~as
\begin{equation}\label{eq:sol_phi}
\mathbb{L}_\kappa^{-1} \;\mathbb{P}_\kappa \; \boldsymbol{\varphi}_+ + \boldsymbol{\varphi}_+ \;=\; \varepsilon^{-1} \mathbb{L}_\kappa^{-1}  \tilde{\mathbf{f}}. 
\end{equation}
Evaluating $\mathbb{P}_\kappa \; \boldsymbol{\varphi}_+$ explicitly yields \eqref{eq:phi_2}.
By a calculation similar to that in Section~\ref{sec:res_cond}, we obtain
\begin{equation}\label{eq:linear_system}
\mathbb{M}_{\kappa,+} \left[
 \begin{array}{llll}
\langle  \boldsymbol{\varphi}_+, \mathbf{e}_1 \rangle  \\
\langle  \boldsymbol{\varphi}_+, \mathbf{e}_2 \rangle
\end{array}
\right] \;=\;
\left[
\begin{array}{llll}
 \langle \mathbb{L}_\kappa^{-1} \tilde{\mathbf{f}}, \mathbf{e}_1 \rangle   \\
 \langle \mathbb{L}_\kappa^{-1} \tilde{\mathbf{f}}, \mathbf{e}_2 \rangle 
\end{array}
\right].
\end{equation}
Recall that the matrix $\mathbb{M}_{\kappa,+}$ has eigenvalues 
$\lambda_{1,+}(k,\varepsilon)$ and $\lambda_{2,+}(k,\varepsilon)$, which are associated with the eigenvectors $v_{1,+}$ and $v_{2,+}$.
By virtue of Lemmas \ref{lem:eigen_M_hat} and~\ref{lem:perturb_eig_M},
\begin{equation}\label{eq:M_inv}
\mathbb{M}_{\kappa,+}^{-1} = \frac{1}{2\lambda_{1,+}}
\left(\left[
\begin{array}{cc}
1 & \frac{1}{1+\eta} \\
1+\eta & 1 
\end{array}
\right] + O(\varepsilon+\kappa^2) \right)
+\frac{1}{2\lambda_{2,+}}
\left(\left[
\begin{array}{cc}
1 & -\frac{1}{1+\eta} \\
-(1+\eta) & 1 
\end{array}
\right] + O(\varepsilon+\kappa^2) \right).
\end{equation}
On the other hand, note that
$$
\tilde{\mathbf{f}}= [-e^{-i\kappa d_0/2},\; -e^{i\kappa d_0/2} ]^T  + O(\kappa\varepsilon).  $$
Thus it follows from Lemma~\ref{lem:L_inv} that
\begin{equation}\label{eq:L_inv_f}
\left[
\begin{array}{llll}
 \langle \mathbb{L}_\kappa^{-1} \tilde{\mathbf{f}}, \mathbf{e}_1 \rangle   \\
 \langle \mathbb{L}_\kappa^{-1} \tilde{\mathbf{f}}, \mathbf{e}_2 \rangle 
\end{array}
\right] = \left(- \left[
\begin{array}{llll}
e^{-i\kappa d_0/2} \\ 
e^{i\kappa d_0/2} 
\end{array}
\right] + O(\kappa\varepsilon) \right)
 (\alpha + O(\varepsilon)),
\end{equation}
The proof is completed by substituting \eqref{eq:M_inv} and \eqref{eq:L_inv_f} into \eqref{eq:linear_system}.  
\end{proof}

\begin{proposition}\label{prop:phi}
Let $\boldsymbol{\varphi} = [ \varphi_1^-, \varphi_1^+, \varphi_2^-, \varphi_2^+ ]^T$ be the solution of the system  $\mathbb{T}\boldsymbol{\varphi}=\varepsilon^{-1}\mathbf{f}$.
If $0<|\kappa| \ll 1$, then 
$\boldsymbol{\varphi}=[ \boldsymbol{\varphi}_{+}+\boldsymbol{\varphi}_{-}, \; \boldsymbol{\varphi}_{+}-\boldsymbol{\varphi}_{-}]^T$, where
$\boldsymbol{\varphi}_{\pm}$ are given in \eqref{eq:phi_2}.
The following asymptotic expansion holds:
\begin{eqnarray*} 
\left[
\begin{array}{llll}
\langle  \varphi_1^-, 1 \rangle  \\
\langle  \varphi_1^+, 1 \rangle  \\
\langle  \varphi_2^-, 1 \rangle  \\
\langle  \varphi_2^+, 1 \rangle 
\end{array}
\right]  \nonumber
&=& -\mu_+(\kappa)\left(\alpha+ O(\varepsilon+\kappa^2) \right)
\left(
\frac{1}{2\lambda_{1,+}} \left[
\begin{array}{cc}
\frac{1}{1+\eta} \\
1  \\
\frac{1}{1+\eta} \\
1 
\end{array}
\right] 
+\frac{1}{2\lambda_{1,-}}\left[
\begin{array}{cc}
\frac{1}{1+\eta} \\
1  \\
-\frac{1}{1+\eta} \\
-1 
\end{array}
\right]  \right) \\
&& +\;\mu_-(\kappa)\left(\alpha+ O(\varepsilon+\kappa^2) \right)
\left(
\frac{1}{2\lambda_{2,+}} \left[
\begin{array}{cc}
\frac{1}{1+\eta} \\
-1  \\
\frac{1}{1+\eta} \\
-1 
\end{array}
\right] 
+\frac{1}{2\lambda_{2,-}}\left[
\begin{array}{cc}
\frac{1}{1+\eta} \\
-1  \\
-\frac{1}{1+\eta} \\
1 
\end{array}
\right]   \right)
\end{eqnarray*} 
$(|\kappa|,\varepsilon\to0)$, where $\alpha:=\langle S^{-1} 1, 1 \rangle$ and $\mu_\pm(\kappa)$  are defined in \eqref{eq:mu_pm}.
\end{proposition}

\subsection{Fano-type transmission anomalies}\label{subsec:anomaly}

Let us consider the field above and below the metallic grating.  Define reflection and transmission coefficients $r^\pm$ and $t^\pm$:
\begin{eqnarray}
&& r^- =  -\frac{\mu_+}{2(1+\eta)} \Lambda_{1,+}
+ \frac{\mu_-}{2(1+\eta)} \Lambda_{2,+} , \quad
 r^+= -\frac{\mu_+}{2} \Lambda_{1,+} - \frac{\mu_-}{2}\Lambda_{2,+}, \label{eq:r_pm} \\
&& t^- = -\frac{\mu_+}{2(1+\eta)}  \Lambda_{1,-}
+ \frac{\mu_-}{2(1+\eta)} \Lambda_{2,-}, \quad
 t^+ = -\frac{\mu_+}{2} \Lambda_{1,-}
- \frac{\mu_-}{2}  \Lambda_{2,-}, \label{eq:t_pm}
\end{eqnarray}
where
\begin{equation}\label{eq:Lambda_pm}
\Lambda_{1,\pm}:=\frac{1}{\lambda_{1,+}} \pm \frac{1}{\lambda_{1,-}}, \quad \Lambda_{2,\pm}:=\frac{1}{\lambda_{2,+}} \pm \frac{1}{\lambda_{2,-}}.
\end{equation}

\begin{lemma}\label{lem:u_far_field}
If $0<|\kappa|\ll 1$, the solution to the scattering problem (\ref{eq:Helmholtz}--\ref{eq:field2}) admits the forms
\begin{align*}
  u_\varepsilon(x)  &=  u^\mathrm{inc} + u^\mathrm{refl}+\varepsilon\left(\alpha+ O(\varepsilon+\kappa^2) \right) 
  \left\{ r^-  \Big[g^\mathrm{e}\big(x,y_1^-\big)+O(\varepsilon)\Big] + r^+ \Big[g^\mathrm{e}\big(x,y_1^+\big)+O(\varepsilon)\Big] \right\} \\
  u_\varepsilon(x)  &= 
    \varepsilon\left(\alpha+ O(\varepsilon+\kappa^2) \right)\left\{ t^- \Big[g^\mathrm{e}\big(x,y_2^-\big)+O(\varepsilon)\Big] 
+ t^+  \Big[g^\mathrm{e}\big(x,y_2^+)\big)+O(\varepsilon)\Big] \right\}
\end{align*}
in $\Omega_1$ and $\Omega_2$ respectively for $x$ far away from the grating, where $y_1^\pm = (\pm d_0/2,1)$ and $y_2^\pm = (\pm d_0/2,0)$ are the centers of the slit apertures $\Gamma^{\pm}_{1,\varepsilon}$ and $\Gamma^{\pm}_{2,\varepsilon}$ respectively. 
\end{lemma}

\begin{proof} From Lemma~\ref{eq:u_eps_formula}, the diffracted field  
$u_\varepsilon^d(x):=u_\varepsilon(x)-(u^\mathrm{inc}+ u^\mathrm{refl})$ above the grating is 
$$ u_\varepsilon^d(x)= \int_{\Gamma^+_{1,\varepsilon} \cup \Gamma^-_{1,\varepsilon}} g^\mathrm{e}(x,y) \frac{\partial u_\varepsilon(y)}{\partial y_2} ds_y  \quad \mbox{in} \; \Omega_1. $$
Thus for $x$ far away from the grating, in the scaled interval $I$, 
\begin{eqnarray*}
u_\varepsilon^d(x)  &=& \varepsilon\int_{-1/2}^{1/2} g^\mathrm{e}\big(x,(-d_0/2+\varepsilon Y,1)\big) \varphi_1^-(Y) \, dY + \varepsilon\int_{-1/2}^{1/2} g^\mathrm{e}\big(x,(d_0/2+\varepsilon Y,1)\big) \varphi_1^+(Y) \, dY \\
 &=& \varepsilon \left( g^\mathrm{e}\big(x,y_1^-\big)+O(\varepsilon) \right) \cdot  \langle \varphi_1^-, 1 \rangle + \varepsilon \left( g^\mathrm{e}\big(x,y_1^+\big)+O(\varepsilon) \right) \cdot  \langle \varphi_1^+, 1 \rangle.
\end{eqnarray*}
By the asymptotic expansion in Proposition~\ref{prop:phi}, we obtain the desired expansion for $u_\varepsilon(x)$.
The wave field for $x\in\Omega_2$ can be obtained similarly. 
\end{proof}

Now we consider the reflected and transmitted wave above and below the grating.  Decompose the Green function $g^\mathrm{e}(x,y)$
into the propagating and exponentially decaying parts $g^\mathrm{e}(x,y)=g_\mathrm{prop}(x,y) + g_\mathrm{exp}(x,y)$. 
Note that for $(\kappa,k)\in D_1$, only one propagating Fourier mode ($n=0$) appears in the Green function.
By substituting the propagating parts of the Green function into the above lemma, we obtain the expansion of the reflected and transmitted fields as follows.
\begin{proposition}\label{prop:u_rt}
If $0<|\kappa|\ll 1$ and $(\kappa,k)\in D_1$, the reflected and transmitted fields admit the forms
$$ u_\varepsilon^r(x) = R(k,\kappa,\varepsilon) e^{i\kappa x_1 + i\zeta_0(x_2-1)} \quad \mbox{and} \quad u_\varepsilon^t(x) =  T(k,\kappa,\varepsilon) e^{i\kappa x_1 - i\zeta_0x_2}, $$
where the reflection and transmission coefficients are
\begin{eqnarray*}
R(k,\kappa,\varepsilon) &=& 1 + \varepsilon \, \tau \, \big(\alpha+ O(\varepsilon+\kappa^2) \big) \cdot
\left( -\mu_+^2 \, \Lambda_{1,+}  +\mu_-^2  \, \Lambda_{2,+} \right), \\
T(k,\kappa,\varepsilon) &=& \varepsilon \, \tau \, \big(\alpha+ O(\varepsilon+\kappa^2) \big) \cdot
\left( -\mu_+^2 \, \Lambda_{1,-} + \mu_-^2 \, \Lambda_{2,-} \right),
\end{eqnarray*}
$\tau(k,\kappa) =-\frac{i}{{2 d\,\zeta_0\,(1+\eta)}}$,  and $\Lambda_{1,\pm}$ and $\Lambda_{2,\pm}$ are defined in \eqref{eq:Lambda_pm}.
\end{proposition}

\medskip

\begin{lemma}\label{lem:disk}
For $r>0$ and a horizontal line $\gamma:=\{ t + ir \,;\, t\in\mathbb{R} \}$ in the complex plane, the set $\{ 1/z \,;\, z\in \gamma \}= D\backslash\{0\}$,
where $D=\left\{ z \,;\, |z+\frac{i}{2r}|=\frac{1}{2r}\right\}$ is a disk.
\end{lemma}
\medskip

We are ready to prove the Fano resonance that occurs in the vicinity of the real resonance frequency $k_*:=\Re \, k_m^{(2)}$ as shown in the transmission graph in Figure~\ref{fig:transmission}. 

\begin{theorem}\label{thm:Fano}
For all $c>0$, define the real interval $I_c:=[k_*-c\kappa^2\varepsilon, k_*+c\kappa^2\varepsilon]$ containing the real resonance frequency~$k_*:=\Re \, k_m^{(2)}$.
There exist a positive number $c$ and frequencies $k_1, k_2 \in I_c$ such that $|T(k_1)| \lesssim \varepsilon$ and $|T(k_2)| \gtrsim 1- \varepsilon$ for $0<|\kappa| \ll1$.
\end{theorem}

\medskip

\begin{rem}
We point out that almost total transmission occurs both near the Fano resonance and near the Fabry-Perot resonance.
In this work, we do not address whether the transmission and reflection are in fact total.
Note that total transmission and/or reflection can be induced by symmetries in certain configurations~\cite{bbcn,chesnelnazarov18,chesnelnazarov19,shipmantu}.
\end{rem}

\begin{proof} We give the proof when $m$ is odd, and the argument is analogous if $m$ is even. 
In view of the asymptotic expansions in Theorem~\ref{thm:asym_eig_perturb} and the explicit expression of $\mu_\pm$ in \eqref{eq:mu_pm},
we see that in the $O(\kappa^2\varepsilon)$ neighborhood of $k_*:=\Re \, k_m^{(2)}$,
$$  \frac{\varepsilon\mu_+^2}{\lambda_{1,+}}=O(1), \;   \frac{\varepsilon\mu_-^2}{\lambda_{2,+}}=O(1), \;  \frac{\varepsilon\mu_+^2}{\lambda_{1,-}}=O(\varepsilon), \;  \frac{\varepsilon\mu_-^2}{\lambda_{2,-}}=O(\kappa^2\varepsilon). $$
Therefore
\begin{equation*}
 R(k) = 1+ \varepsilon \, \alpha\tau \cdot
\left( - 
\frac{\mu_+^2}{\lambda_{1,+}} +
\frac{\mu_-^2}{\lambda_{2,+}}   
 \right) +  O(\varepsilon), 
\quad
T(k) =\varepsilon \, \alpha\tau \cdot
\left( -
\frac{\mu_+^2}{\lambda_{1,+}} + 
\frac{\mu_-^2}{\lambda_{2,+}}   
 \right) +  O(\varepsilon),
\end{equation*}
and it follows that $R=1+T+O(\varepsilon)$.
By substituting into the equation for the conservation of energy $|R|^2+|T|^2=1$, we obtain
$|T(k)+1|^2 + |T(k)|^2 \;=\; 1 + O(\varepsilon)$.
This shows that, for fixed $\varepsilon$, the trajectory $\gamma_\varepsilon$ of the transmission coefficient $T(k)$ ($k\in I_c$) on the complex plane lies close to the fixed circular trajectory 
$\gamma_0$ with radius $\frac{1}{2}$ centered at $(-\frac{1}{2},0)$ (cf. Figure~\ref{fig:tran_complx}, right).
Namely,
\begin{equation}\label{eq:gamma}
\gamma_\varepsilon=\gamma_0 +O(\varepsilon),  \quad\mbox{with} \;  \gamma_0 \subset \big\{ \textstyle z \,;\, |z+\frac{1}{2}|=\frac{1}{2}\big\}.
\end{equation}
In fact, the assertion of the theorem holds as long as 
\begin{equation}\label{claim}
[0, \pi ] \; \mbox{or} \; [\pi, 2\pi] \subset \left\{ \arg \left(z+\frac{1}{2}\right);  \; z \in \gamma_0 \right\}.
\end{equation}

\begin{figure}
\begin{center}
\hspace*{-1.5cm}
\includegraphics[height=5.5cm]{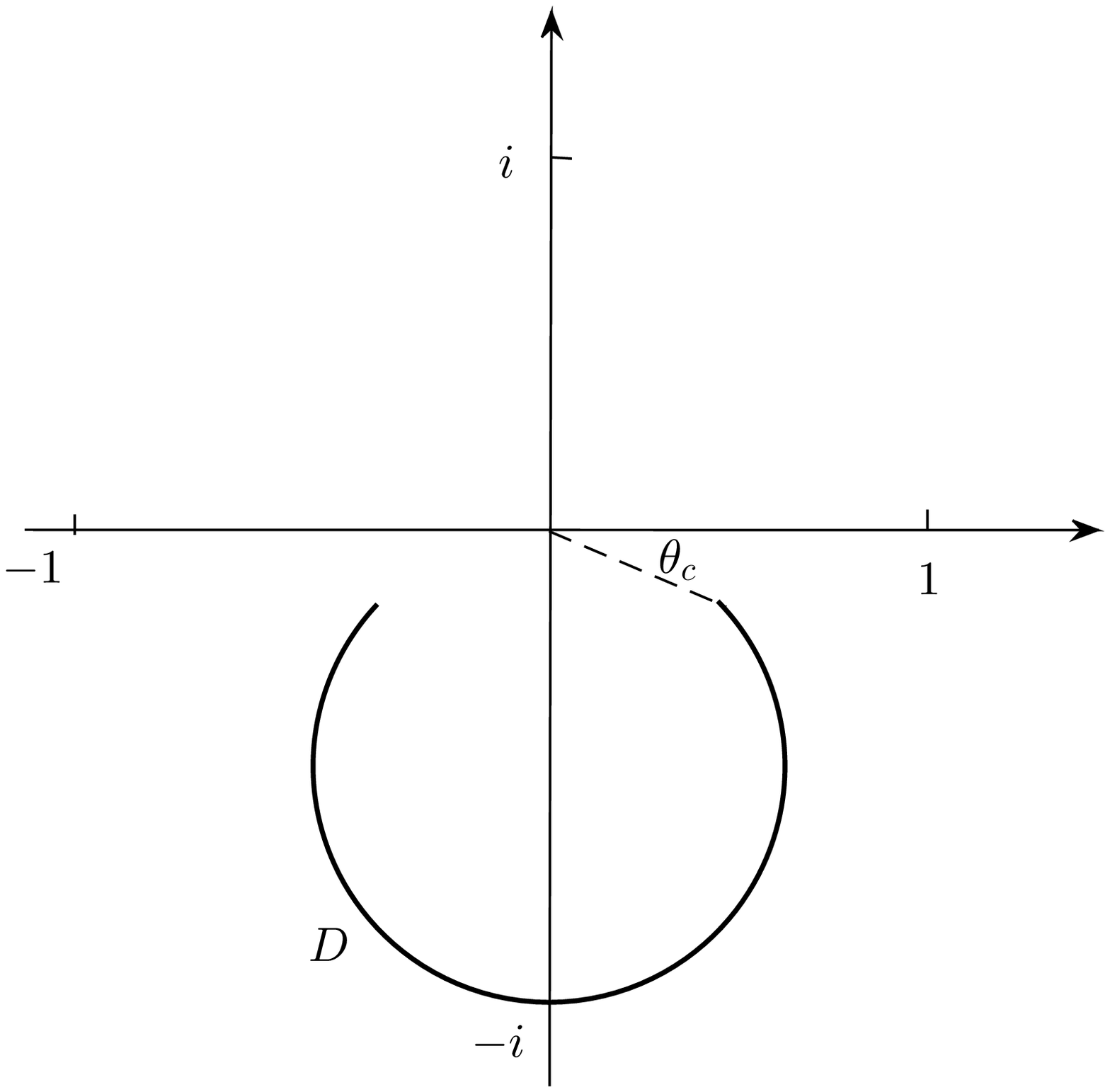}
\hspace*{-2cm}
\includegraphics[height=5.5cm]{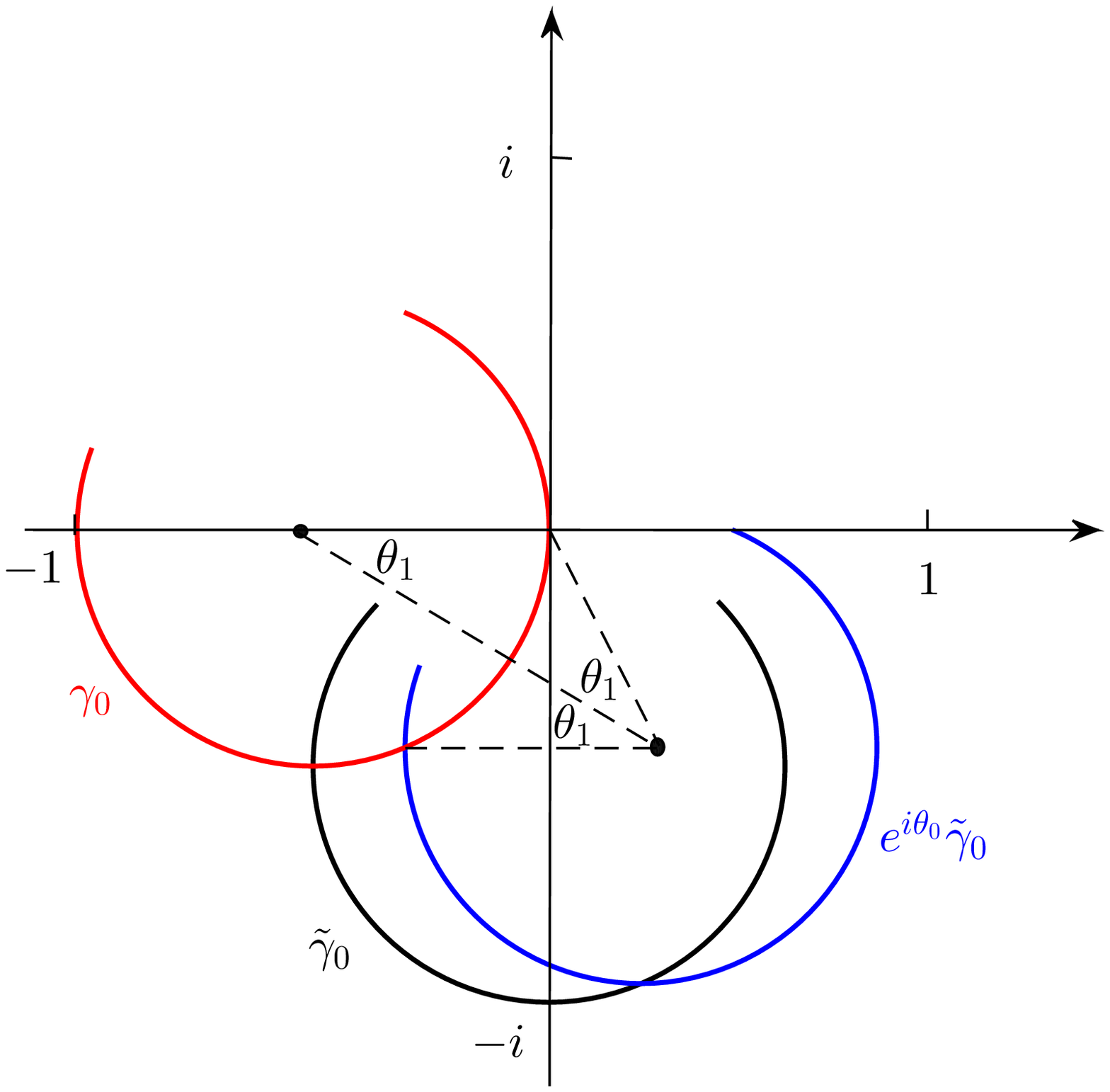}
\caption{Left: The curve $\tilde \gamma_0$ lies on the circle $\big\{ z \,;\, |z+\frac{i}{2}|=\frac{1}{2}\big\}$. Right: The curves $\gamma_0$, $\tilde\gamma_0$ and $e^{i \theta_0}\tilde\gamma_0$.  $\gamma_0$ lies on the circle $\big\{ z \,;\, |z+\frac{1}{2}|=\frac{1}{2}\big\}$.}\label{fig:tran_complx}
\end{center}
\end{figure}
To show this, write $T(k)= t_1(k)  + t_2(k)  +  O(\varepsilon)$, where
\begin{equation}\label{eq:transmission1}
 t_1(k) = - \alpha\tau(k)  \frac{\varepsilon \mu_+^2(k)}{\lambda_{1,+}(k)}, \quad t_2(k)= \alpha \tau(k)  \frac{ \varepsilon \mu_-^2(k)}{\lambda_{2,+}(k)}.
\end{equation}
From a perturbation argument parallel to the proof of Theorem~\ref{thm:asym_eig_perturb}, it is known that $\Re \, \lambda_{2,+}(k)$ attains a root $\tilde k_*$
in the vicinity of $k_*$. Since $\lambda_{2,+}(k_*)=O(\kappa^2\varepsilon)$, we deduce that  $|\tilde k_*-k_*|=O(\kappa^2\varepsilon)$ and $\Im \, \lambda_{2,+}(\tilde k_*)=O(\kappa^2\varepsilon)$. 
Now expand all the terms of \eqref{eq:transmission1} in the $O(\kappa^2\varepsilon)$ neighborhood of $\tilde k_*$:
\begin{eqnarray*}
&&  \tau(k)=\tau(\tilde k_*) + O(\kappa^2\varepsilon), \quad \mu_\pm(k) = \mu_\pm(\tilde k_*) + O(\kappa^2\varepsilon), \\
&&  \lambda_{1,+}(k) = \lambda_{1,+}(\tilde k_*) + O(\kappa^2\varepsilon), \quad \lambda_{2,+}(k) = c_1 (k-\tilde k_*) + i \, c_2 \kappa^2\varepsilon + O(\kappa^2\varepsilon^2),
\end{eqnarray*}
where $c_1$ and $c_2$ are real-valued constants and $c_2>0$.  By setting $k-\tilde k_*=s \cdot \kappa^2\varepsilon$, it follows that
\begin{equation*}
T(k) = t_1 (\tilde k_*)+ \alpha \tau(\tilde k_*) \frac{ \varepsilon \mu_-^2(\tilde k_*)}{ c_1 (k-\tilde k_*) + i \, c_2 \kappa^2\varepsilon}  +  O(\varepsilon)
=  t_1 (\tilde k_*)+  \frac{e^{i \theta_0}}{ \hat c_1 s + i \, \hat c_2 }  +  O(\varepsilon),
\end{equation*}
in which
$$c_0:= \alpha \tau(\tilde k_*)  \frac{\mu_-^2(\tilde k_*)}{\kappa^2}=O(1), \quad  \hat c_1=\frac{c_1}{|c_0|}, \quad  \hat c_2=\frac{c_2}{|c_0|}, \quad  \theta_0= \arg c_0.$$
 From Lemma~\ref{lem:disk}, we deduce that 
the trajectory of $T(k)$ for $k \in I_c$ is given by
\begin{equation}\label{eq:transmission2}
\gamma_\varepsilon= t_1 (\tilde k_*) + e^{i \theta_0} \tilde\gamma_0 +O(\varepsilon), \quad \mbox{where}  \; \tilde\gamma_0 \subset \big\{ \textstyle z \,;\, |z+\frac{i}{2\hat c_2}|=\frac{1}{2\hat c_2}\big\}.
\end{equation}
In addition,
$\big[\pi+\theta_c,2\pi-\theta_c\big]  \subset \big\{ \arg z; \; z \in \tilde\gamma_0 \big\}$
for certain $\theta_c\in(0,\pi/2)$ depending on the constant $c$  (see Figure~\ref{fig:tran_complx}, left).

A combination of \eqref{eq:gamma} and \eqref{eq:transmission2} leads to the relation
$ \gamma_0 = t_1 (\tilde k_*) + e^{i \theta_0} \tilde\gamma_0$. 
Geometrically, $\gamma_0$ is obtained from a rotation and translation of the curve $\tilde\gamma_0$ as shown in Figure~\ref{fig:tran_complx} (right),
and it follows that $\hat c_2=1$.
A direct calculation shows that $|\Im\, t_1 (\tilde k_*)|/|\Re\, t_1 (\tilde k_*)|=O(1)$ and is nonzero. By letting
$\theta_1 : =\tan^{-1}(|\Im\, t_1 (\tilde k_*)|/ |\Re\, t_1 (\tilde k_*)|)$ 
and choosing sufficiently large $c$ such that $\theta_c<\theta_1$ and $\theta_c<\frac{\pi}{2}-\theta_1$, the claim 
\eqref{claim} holds and the proof is complete.
\end{proof}

\begin{rem}
From the proof, $\gamma_0$ is the trajectory for the leading-order term of the transmission coefficient on the complex plane, and 
using \eqref{claim}, 
the graph $\{ |z|; \; z \in \gamma_0 \}$ demonstrates an asymmetric line shape with respect to the frequency for $k \in I_c$.
\end{rem}

\subsection{Field enhancement}\label{subsec:amplification}
Fano resonance is usually associated with field amplification around the resonance frequencies \cite{abeynanda, fan, shipman13}.
This also applies to the periodic structure considered here.
See Figure~\ref{fig:u_slit} for a plot of the field inside the slits at Fano resonance frequencies. 

We investigate the field enhancement at frequencies around the real part $\Re\, k_m^{(2)}$ of a complex resonance that is the perturbation of a real eigenvalue.
It is known that the field amplification at Fabry-Perot resonance frequencies $\Re \, k_m^{(1)}$  is of order $O(1/
\varepsilon)$ \cite{lin_zhang18_1}. 
As shown below, field amplification with an order of $O(1/(\kappa\varepsilon))$ occurs at the Fano resonance frequencies $\Re \, k_m^{(2)}$, which is much stronger than that of Fabry-Perot resonance. This results in more complicated scattering behavior, as the field enhancement depends on both small $\varepsilon$ and $\kappa$.

\begin{figure}[!htbp]
\begin{center}
\includegraphics[height=2.5cm]{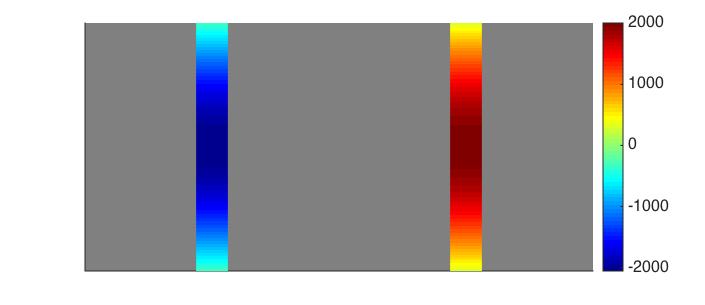} \hspace*{-20pt}
\includegraphics[height=2.5cm]{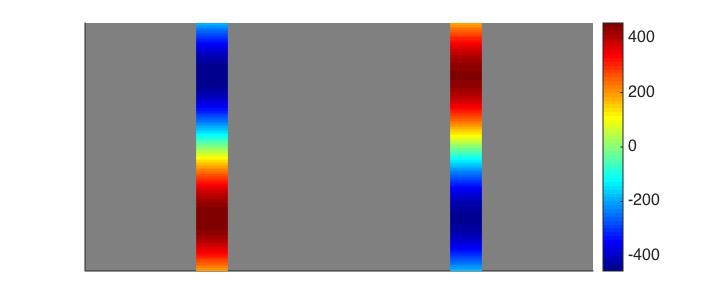}
\vspace*{-15pt}
\caption{The wave field inside the slits $S_{\varepsilon}^{0,\pm}$ at the first two  Fano resonance frequencies $\Re \, k_1^{(2)}$ and $\Re \, k_2^{(2)}$. $d=1$, $d_0=0.4$, $\varepsilon=0.05$, $\kappa=0.1$. }\label{fig:u_slit}
\end{center}
\end{figure}

In what follows, for clarity we focus on the field amplification inside the slit only. The same amplification order
can be obtained near the slit aperturs and in the far-field following the method in \cite{lin_zhang18_1}, here we skip the calculations for brevity.
Since $u_\varepsilon$ is quasi-periodic, we analyze the field in the reference slit $S_{\varepsilon}^{(0)}$.  

\begin{lemma}\label{lem:u_slit}
In the slit region $\tilde S_{\varepsilon}^{0,\pm} :=\{ x\in S_{\varepsilon}^{0,\pm} \;;\; x_2 \gg \varepsilon , 1- x_2 \gg \varepsilon \} $, the solution $u_\varepsilon(x)$ of the scattering problem (\ref{eq:Helmholtz}--\ref{eq:field2}) admits the following expansion
$$
u_\varepsilon(x) \;=\; -\frac{\alpha+ O(\varepsilon+\kappa^2)}{k\sin k} \Big( r^\pm \cos kx_2 + t^\pm \cos k(1-x_2) \Big) + O\left(e^{-1/\varepsilon} \right),
$$
where the coefficients $r^\pm$ and $t^\pm$ are given in \eqref{eq:r_pm} and \eqref{eq:t_pm}.
\end{lemma}
 \begin{proof}
The field $u_\varepsilon$ satisfies the Helmholtz equation in $S_{\varepsilon}^{0,\pm}$ with homogeneous Neumann boundary conditions on the slit walls, and thus it admits the expansion
\begin{equation*}
u_\varepsilon(x) \;=\; a_{0}^\pm \cos kx_2 + b_0^\pm \cos k(1-x_2) + \sum_{m\geq 1} \left(a_m^\pm e^{-k_{2,\varepsilon}^{(m)}x_2} +  b_m^\pm e^{-k_{2,\varepsilon}^{(m)}(1-x_2)} \right) \cos \frac{m\pi \tilde x_{1,\varepsilon}^\pm}{\varepsilon},
\end{equation*}
where $\tilde x_{1,\varepsilon}^\pm = x_1 \mp d_0/2 + \varepsilon/2$ and $k_{2,\varepsilon}^{(m)}=\sqrt{(m\pi/\varepsilon)^2-k^2}$.
Taking the derivative of the above expansion with respect to $x_2$ and integrating over the slit apertures yields
\begin{equation*}
- a_0^\pm k\sin k = \frac{1}{\varepsilon} \int_{\Gamma^\pm_{1,\varepsilon}} \frac{\partial u_\varepsilon}{\partial x_2} ds 
= \langle \varphi_1^\pm, 1 \rangle, \quad
b_0^\pm k \sin k = \frac{1}{\varepsilon} \int_{\Gamma^\pm_{2,\varepsilon}} \frac{\partial u_\varepsilon}{\partial x_2} ds 
= -\langle \varphi_2^\pm, 1 \rangle.
 \end{equation*}
Applying Proposition~\ref{prop:phi}, we obtain the expansion coefficients $a_0^\pm$ and $b_0^\pm$ as follows:
\begin{equation}\label{a0b0}
a_0^\pm  = -\frac{r^\pm (\alpha+ O(\varepsilon+\kappa^2))}{k\sin k},  \quad
b_0^\pm  = -\frac{t^\pm (\alpha+ O(\varepsilon+\kappa^2))}{k\sin k}.
\end{equation}
For $m\ge1$, the coefficients $a_m$ and $b_m$ can be obtained similarly by taking the inner product of 
$\partial_{x_2} u$ with $\cos \frac{m\pi \tilde x_1^\pm }{\varepsilon}$ over the slit apertures. In view of Propsosition
\ref{prop:phi}, a direct estimate leads to
\begin{equation}\label{ambm}
\abs{a_m} \le O(1/\sqrt{m}) , \quad \abs{b_m} \le O(1/\sqrt{m}), \quad \mbox{for}\; \,m\ge1.
\end{equation}
The proof is complete.  
\end{proof}

Now the shape of resonant wave modes in the slits  and their enhancement orders at the Fano resonance frequency $k=\Re \, k_m^{(2)}$
are characterized in the following theorem.

\begin{theorem}\label{thm:u_slit_res}
In the slit region $\tilde S_{\varepsilon}^{0,\pm} :=\{ x\in S_{\varepsilon}^{0,\pm} \;;\; x_2 \gg \varepsilon , 1- x_2 \gg \varepsilon \} $, the solution $u_\varepsilon(x)$ of the scattering problem (\ref{eq:Helmholtz}--\ref{eq:field2}) admits the following asymptotic form at the resonant frequencies $k= \Re \, k_m^{(2)}$.
\begin{align*}
   u_\varepsilon(x) &=\; \left[\pm \frac{c_\mathrm{odd}}{\kappa\varepsilon} + O\left(\frac{1}{\varepsilon}\right) \right] \cos (k(x_2-1/2)) + O(1) ,  \quad
 (k=\Re \, k_m^{(2)}, \; m\;\text{even})  \\
   u_\varepsilon(x) &=\; \left[\pm \frac{c_\mathrm{even}}{\kappa\varepsilon} + O\left(\frac{1}{\varepsilon}\right) \right] \sin (k(x_2-1/2)) + O(1),  \quad
 (k=\Re \, k_m^{(2)}, \; m\;\text{odd})
\end{align*}
($|\kappa|,\varepsilon\to0$), in which $c_\mathrm{odd}$ and $c_\mathrm{even}$ are certain constants independent of $\varepsilon$ and $\kappa$.
\end{theorem}

\begin{proof} We only perform the calculations when $m$ is odd, and the calculations for even $m$ are similar.
From Lemma~\ref{lem:u_slit} and the explicit expressions (\ref{eq:r_pm}--\ref{eq:t_pm}) for the coefficients $r^\pm$ and $t^\pm$, we obtain that in the regions $S_{\varepsilon}^{0,-}$
and $ S_{\varepsilon}^{0,+}$,
\begin{eqnarray}
u_\varepsilon(x) &=& -\frac{\alpha+ O(\varepsilon+\kappa^2)}{k\sin k} \Big( r^- \cos kx_2 + t^- \cos k(1-x_2) \Big) + O\left(e^{-1/\varepsilon} \right), \nonumber \\
&=&
 \Big(\alpha+ O(\varepsilon+\kappa^2)\Big)
\left(\frac{\mu_+}{\lambda_{1,+}}  - \frac{\mu_-}{\lambda_{2,+}} \right) \frac{\cos kx_2 +\cos k(1-x_2)}{2(1+\eta) k\sin k} \label{eq:u_slit_res-} \\
 && +
 \Big(\alpha+ O(\varepsilon+\kappa^2)\Big)
\left(\frac{\mu_+}{\lambda_{1,-}}  - \frac{\mu_-}{\lambda_{2,-}} \right) \frac{\cos kx_2 -\cos k(1-x_2)}{2(1+\eta) k\sin k}
+ O\left(e^{-1/\varepsilon} \right), \nonumber 
\end{eqnarray}
\begin{eqnarray}
u_\varepsilon(x) &=& -\frac{\alpha+ O(\varepsilon+\kappa^2)}{k\sin k} \Big( r^+ \cos kx_2 + t^+ \cos k(1-x_2) \Big) + O\left(e^{-1/\varepsilon} \right), \nonumber \\
&=&
 \Big(\alpha+ O(\varepsilon+\kappa^2)\Big)
\left(\frac{\mu_+}{\lambda_{1,+}}  + \frac{\mu_-}{\lambda_{2,+}} \right) \frac{\cos kx_2 +\cos k(1-x_2)}{2k\sin k} \label{eq:u_slit_res+} \\
 && +
 \Big(\alpha+ O(\varepsilon+\kappa^2)\Big)
\left(\frac{\mu_+}{\lambda_{1,-}}  + \frac{\mu_-}{\lambda_{2,-}} \right) \frac{\cos kx_2 -\cos k(1-x_2)}{2k\sin k}
+ O\left(e^{-1/\varepsilon} \right)\nonumber 
\end{eqnarray}
respectively.
From the asymptotic expansions in Theorem~\ref{thm:asym_eig_perturb} and the defintion of $\mu_\pm$ in \eqref{eq:mu_pm},
we see that at resonant frequencies $k= \Re \, k_m^{(2)}$,
\begin{equation}\label{eq:mu_over_lambda12_asy}
\frac{1}{\lambda_{1,+}}=O\left(\frac{1}{\varepsilon}\right), \quad \frac{1}{\lambda_{2,+}} = O\left(\frac{1}{\kappa^2 \varepsilon}\right)
\quad \mbox{and} \quad  \mu_+ = 1+ O(\kappa), \quad \mu_-  = O(\kappa).
\end{equation}
On the other hand, in view of Lemmas \ref{lem:eigen_M_hat} and \ref{lem:perturb_eig_M}, 
\begin{equation}\label{eq:lambda_12_asy}
\lambda_{1,-}=  \frac{ (\cos k-1) \alpha }{k \sin k } + O(\varepsilon) 
\quad \mbox{and} \quad
\lambda_{2,-}=  \frac{ (\cos k-1) \alpha }{k \sin k } + O(\varepsilon).
\end{equation}
We obtain the desired expansions by substituting \eqref{eq:mu_over_lambda12_asy}--\eqref{eq:lambda_12_asy} into \eqref{eq:u_slit_res-}--\eqref{eq:u_slit_res+}. 
\end{proof}

\appendix \section{Proof of Lemmas \ref{lem:S-1} and \ref{lem:S-2}}

\noindent\textbf{Proof of Lemmas \ref{lem:S-1}} \quad
Recall that the kernel of $S$ is given by \eqref{eq:rho}.
Note that $\ln\abs{\sin \left(\frac{\pi t}{2}\right)}$ is a periodic function with period $2$.
Setting $\tilde X=-X$ and $\tilde Y=-Y$, it follows that $\rho(X,Y)=\rho(\tilde X, \tilde Y)$.
If $\tilde{\varphi}(X)=\varphi(-X)$, then
\begin{equation}\label{eq:equality_S}
[S \tilde\varphi](X) = \int_{-\frac 12}^{\frac 12} \rho(X,Y) \varphi(-Y) \, dY = \int_{-\frac 12}^{\frac 12}  \rho(\tilde X, \tilde Y) \varphi(\tilde Y) \, d\tilde Y = [S \varphi](\tilde X).
\end{equation}
The kernels of $S^{\infty}_0$ and $\tilde S^{\infty}$ are given by (cf.~\eqref{eq:rho_infty}--\eqref{eq:rho_infty_tilde} and \eqref{eq:re_kappa0})
\begin{eqnarray*}
 \rho_{\infty}(0; X,Y) &=&  \hat r_\mathrm{e}(|X-Y|) + r_{\mathrm{i},1}(\varepsilon; |X-Y|) + r_{\mathrm{i},2}(\varepsilon;|X+Y+1|),  \\
 \tilde\rho_{\infty}(X,Y) &=&  \tilde r_{\mathrm{i},1}(\varepsilon; |X-Y|) + \tilde r_{\mathrm{i},2}(\varepsilon;|X+Y+1|). 
\end{eqnarray*}
If $\tilde{\varphi}(X)=\varphi(-X)$,
in view of the periodicity of the functions $r_{\mathrm{i},2}(t)$ and $\tilde r_{\mathrm{i},2}(t)$ (cf.~\eqref{eq:r12}),  a parallel 
derivation as in \eqref{eq:equality_S} yields
\begin{equation}\label{eq:equality_S_infty}
[S_0^\infty \tilde\varphi](X) = [S_0^\infty \varphi](\tilde X) \quad \mbox{and} \quad  [\tilde S^\infty \tilde\varphi](X) = [\tilde S^\infty \varphi](\tilde X).
\end{equation}
Finally, to show $[S^{\infty,+}_0\tilde{\varphi}](X) = [S^{\infty,-}_0\varphi](-X)$, by \eqref{eq:re_kappa0}
the kernel of $S^{\infty,\pm}_0$ is given by $\hat\rho( |\pm d_0+\varepsilon(X-Y)|)$ for the real-valued function $\hat\rho$. Hence,
\begin{eqnarray*}
[S^{\infty,+}_0\tilde{\varphi}](X)  &=& \int_{-\frac 12}^{\frac 12} \hat\rho(|d_0+\varepsilon(X-Y)|) \varphi(-Y) \, dY  \\
&=& \int_{-\frac 12}^{\frac 12} \hat\rho(|-d_0+\varepsilon(\tilde X-\tilde Y)|) \varphi(\tilde Y) \, d\tilde Y = [S^{\infty,-}_0\varphi](\tilde X).
\end{eqnarray*}

\noindent\textbf{Proof of Lemmas \ref{lem:S-2}} \quad
The proof of (1) can be found in \cite{eric10}. The invertibility of the operator 
$S+S^{\infty}_\kappa+\tilde S^{\infty}$
is evident from (1) and the fact that $\| S_\kappa^{\infty}\|  \lesssim \varepsilon$ and $ \| \tilde S^{\infty}\|  \lesssim e^{-1/\varepsilon} $.
Let $\varphi$ and $\hat\varphi$ satisfy 
$(S+S^{\infty}_0+\tilde S^{\infty})\varphi = g$ and $(S+S^{\infty}_0+\tilde S^{\infty})\hat{\varphi}=\tilde{g}$,
in which $\tilde{g}(X)=g(-X)$.
A combination of \eqref{eq:equality_S} and \eqref{eq:equality_S_infty} leads to 
$$ [(S+S^{\infty}_0+\tilde S^{\infty})\tilde{\varphi}](X) = [(S+S^{\infty}_0+\tilde S^{\infty})\varphi](\tilde X) = g(-X). $$
The assertion (3) holds by the uniqueness of the solution to the integral equation.

\bibliography{references}

\begin{thebibliography}{99}

\bibitem{abeynanda}
{\sc G. S. Abeynanda and S. P. Shipman},
{\em Dynamic Resonance in the High-Q and Near-Monochromatic Regime},
Proc. 16th Int. Conf. on Math. Meth. in Elec. Theory (MMET16) (2016).

\bibitem{bao95}
{\sc G. Bao, D. Dobson, and Cox}, {\em Mathematical studies in rigorous grating theory}, J. Opt. Soc. Amer. A, \textbf{12} (1995), 1029--1042.

\bibitem{bonnet_starling94}
{\sc  A. Bonnet-Bendhia and F. Starling}, {\em Guided waves by electromagnetic gratings and non-uniqueness examples for the diffraction problem},
Math. Meth. Appl. Sci., \textbf{17} (1994), 305--338.

\bibitem{bbcn}
{\sc A.-S. Bonnet-Ben Dhia, L. Chesnel, and S.A. Nazarov}, {\em Perfect transmission invisibility for waveguides with sound hard walls}, J. Math. Pures Appl., in press (2017).

\bibitem{chesnelnazarov18}
{\sc L. Chesnel, S.A. Nazarov}, {\em Non reflection and perfect reflection via Fano resonance in waveguides}, Comm. Math. Sci., \textbf{16} No. 7 (2018) 1779--1800.

\bibitem{chesnelnazarov19}
{\sc L. Chesnel, S.A. Nazarov}, {\em Exact zero transmission during the Fano resonance phenomenon in non symmetric waveguides}, hal-02189311f (2019).

\bibitem{eric10-2}
{\sc  J. F. Babadjian, E. Bonnetier and F. Triki},
{\em Enhancement of electromagnetic fields caused by interacting subwavelength cavities},
Multiscale Model. Simul., \textbf{8} (2010), 1383--1418. 

\bibitem{eric10}
{\sc  E. Bonnetier and F. Triki},
{\em Asymptotic of the Green function for the diffraction by a perfectly
conducting plane perturbed by a sub-wavelength rectangular cavity}, Math. Meth. Appl.
Sci., \textbf{33} (2010), 772--798.


\bibitem{nazarov18}
{\sc L. Chesnel and S. A. Nazarov},
{\em Non reflection and perfect reflection via Fano resonance in waveguides},
Comm. Math. Sci.,
\textbf{16}(7) (2018),
DOI: 10.4310/CMS.2018.v16.n7.a2

\bibitem{collin}
{\sc R. Collin}, Field Theory of Guided Waves (2nd Edition), Wiley-IEEE Press, 1990.


\bibitem{ebbesen98} 
{\sc  T. W. Ebbesen, H. J. Lezec, H. F. Ghaemi, T. Thio, and P. A. Wolff}, 
{\em Extraordinary optical transmission through sub-wavelength hole arrays}, Nature, \textbf{391} (1998), 667--669.

\bibitem{eistenstat} 
{\sc S. Eisenstat and I. Ipsen}, {\em Three absolute perturbation bounds for matrix eigenvalues imply relative bounds}, SIAM J. Matr. Anal. Appl., \textbf{20} (1998), 149--158.

\bibitem{fano}
{\sc U. Fano}, {\em Effects of configuration interaction on intensities and phase shifts}, Phy. Rev., \textbf{124} (1961), 1866.

\bibitem{fan}
{\sc S. Fan, W. Suh, and J. D. Joannopoulos},
{\em Temporal coupled-mode theory for the {F}ano resonance in optical resonators},
J. Opt. Soc. Am. A, \textbf{20}(3) (2003) 569--572.

\bibitem{garcia10} 
{\sc  F. J. Garcia-Vidal, L. Martin-Moreno, T. W. Ebbesen, and L. Kuipers}, 
{\em Light passing through subwavelength apertures}, Rev. Modern Phys., \textbf{82} (2010), 729--787.

\bibitem{hsu1}
{\sc C. W. Hsu, B. Zhen, S. L. Chua, S. Johnson, J. Joannopoulos, and M. Soljacic},
{\em Bloch surface eigenstates within the radiation continuum},
Light: Science App., 2(e84 doi:10:1038/Isa.2013.40), 2013.

\bibitem{hsu2}
{\sc C. W. Hsu, B. Zhen, J. Lee, S. L. Chua, S. Johnson, J. Joannopoulos, and M. Soljacic},
{\em Observation of trapped light within the radiation continuum},
Nature, (doi:10.1038/nature12289), July 2013.

\bibitem{hsu3}
{\sc C. W. Hsu, \textit{et al.}}, {\em Bound states in the continuum}, Nat. Rev. Mater., \textbf{1} (2016), 16048.

\bibitem{kress}
{\sc R. Kress}, {\em Linear Integral Equations}, Applied Mathematical Sciences, vol. 82, Springer-Verlag, Berlin, 1999.

\bibitem{lin14} 
{\sc  J. Lin, S.-H. Oh, H.-M. Nguyen, and F. Reitich},  {\em Field enhancement and saturation of millimeter waves inside a metallic nanogap}, Opt. Express, \textbf{22} (2014), pp. 14402--14410. 

\bibitem{lin15} 
{\sc  J. Lin and F. Reitich}, {\em Electromagnetic field enhancement in small gaps: a rigorous mathematical theory}, SIAM J. Appl. Math., \textbf{75} (2015), 2290--2310.

\bibitem{lin_zhang17}
{\sc  J. Lin and H. Zhang}, {\em Scattering and field enhancement of a perfect conducting narrow slit}, SIAM J. Appl. Math.,  (2017), 951--976.

\bibitem{lin_zhang18_1}
{\sc  J. Lin and H. Zhang}, {\em Scattering by a periodic array of subwavelength slits I: field enhancement in the diffraction regime}, Multiscale Model. Simul., \textbf{16} (2018), 922--953.

\bibitem{lin_zhang18_2}
{\sc  J. Lin and H. Zhang}, {\em Scattering by a periodic array of subwavelength slits II: surface bound states, total transmission and field enhancement in the homogenization regimes},
 Multiscale Model. Simul., \textbf{16} (2018), 954--990.

\bibitem{limonov}
{\sc M. Limonov, \textit{et al.}},  {\em Fano resonances in photonics}, Nat. Photon. \textbf{11} (2017): 543.
 
\bibitem{linton98}
{\sc C. Linton}, {\em The Green function for the two-dimensional Helmholtz equation in periodic domains}, J. Eng. Math. \textbf{33} (1998), 377--401.

\bibitem{lukyanchuk}
{\sc B. Luk'yanchuk, \textit {et al.}}, {\em The Fano resonance in plasmonic nanostructures and metamaterials}, Nat. Mat. \textbf{9} (2010): 707.

\bibitem{shipmantu}
{\sc Stephen P. Shipman and Hairui Tu}, {\em Total Resonant Transmission and Reflection by Periodic Structures}, SIAM J. Appl. Math., \textbf{72}, No. 1 (2012) 216--239.

\bibitem{shipman05}
{\sc S. P. Shipman and S. Venakides},  {\em Resonant transmission near nonrobust periodic slab modes}, Phy. Rev. E, \textbf{71} (2005), 026611.

\bibitem{shipman07}
{\sc S. P. Shipman and D. Volkov}, {\em Guided modes in periodic slabs: existence and nonexistence}, SIAM J. Appl. Math.,  \textbf{67} (2007), 687--713.

\bibitem{shipman10}
{\sc S. P. Shipman}, {\em Resonant scattering by open periodic waveguides},  Chapter 2 in Wave Propagation in Periodic Media: Analysis, Numerical Techniques and Practical Applications, M. Ehrhardt, ed., E-Book Series PiCP, Bentham Science Publishers, Vol. 1 (2010).

\bibitem{shipman13}
{\sc S. P. Shipman and A. T. Welters},
{\em Resonant electromagnetic scattering in anisotropic layered media},
Journal of Mathematical Physics, \textbf{54}(10) (2013) 103511:1--40.

 \bibitem{lu1}
{\sc L. Yuan and Y-Y. Lu}, {\em Propagating Bloch modes above the lightline on a periodic array of cylinders},  J. Phy. B: Atomic, Molecular and Optical Physics, \textbf{50} (2017), 05LT01.

\bibitem{lu2}
{\sc L. Yuan and Y-Y. Lu}, {\em Bound states in the continuum on periodic structures: perturbation theory and robustness}, Opt. Lett., \textbf{42} (2017), 4490--4493.

\end{thebibliography}

\end{document}